\documentclass[12pt]{amsart}
\usepackage[utf8]{inputenc}

\usepackage{hyperref}
\hypersetup{nesting=true,debug=true,naturalnames=true}
\usepackage{graphicx,amssymb,upref}

\usepackage{amsmath,amsthm}
\usepackage{amsfonts}
\usepackage{latexsym}
\usepackage{amsbsy}
\usepackage{amssymb,mathscinet}
\DeclareMathOperator\supp{supp}
\numberwithin{equation}{section}

\usepackage{amsfonts,amssymb,amscd,amsmath,enumerate,verbatim,calc,enumerate}
\theoremstyle{plain}
\newtheorem{theorem}{Theorem}[section]
\newtheorem{corollary}[theorem]{Corollary}
\newtheorem{lemma}[theorem]{Lemma}

\newtheorem{question}[theorem]{Question}

\newtheorem{definition}[theorem]{Definition}

\newtheorem{remark}[theorem]{Remark}

\newcommand{\wV}{\wt{V}}
\newcommand{\be}{\mathbb E}
\newcommand{\bn}{\mathbb N}
\newcommand{\Nk}{\bn_0^k}
\newcommand{\ot}{\otimes}
\newcommand {\id} {{\textrm{id}}}
\newcommand{\wt}{\widetilde}
\newcommand{\wT}{\wt{T}}

\newtheorem{thm}{Theorem}[section]

\theoremstyle{defn}

\theoremstyle{rem}

\numberwithin{equation}{section}

\date{\today}
\author[Rohilla]{Azad Rohilla \textsuperscript{*}}
\address{Centre for Mathematical and Financial Computing, Department of Mathematics, The LNM Institute of Information Technology, Rupa ki Nangal, Post-Sumel, Via-Jamdoli
	Jaipur-302031,
	(Rajasthan) INDIA}
\email{18pmt005@lnmiit.ac.in}
\author[Trivedi]{Harsh Trivedi}
\address{Centre for Mathematical and Financial Computing, Department of Mathematics, The LNM Institute of Information Technology, Rupa ki Nangal, Post-Sumel, Via-Jamdoli
	Jaipur-302031,
	(Rajasthan) INDIA}
\email{harsh.trivedi@lnmiit.ac.in, trivediharsh26@gmail.com}
\author[Veerabathiran]{Shankar Veerabathiran }
\address{Indian Statistical Institute, Statistics and Mathematics Unit, 8th Mile, Mysore Road,
	Bangalore, 560059, India}
\email{shankarunom@gmail.com}
\thanks{*corresponding author}

\begin{document}
	
	\title[Beurling quotient subspaces for  covariant representations of product systems]
	{Beurling quotient subspace for  covariant representations of product systems}

	\date{\today}

\begin{abstract}Let $(\sigma, V^{(1)}, \dots, V^{(k)})$ be a  pure doubly commuting isometric representation  of  $\mathbb{E}$ on a Hilbert space $\mathcal{H}_{V}.$ A   $\sigma$-invariant subspace $\mathcal{K}$  is said to be {\it Beurling quotient subspace} of $\mathcal{H}_{V}$ if there exist a Hilbert space $\mathcal{H}_W,$ a pure doubly commuting isometric representation  $(\pi, W^{(1)}, \dots, W^{(k)})$ of $\mathbb{E}$ on  $\mathcal{H}_W$ and an isometric multi-analytic operator $M_\Theta:{\mathcal{H}_W} \to \mathcal{H}_{V}$ such that
	\begin{equation*}
		\mathcal{K}=\mathcal{H}_{V}\ominus M_{\Theta}\mathcal{H}_W,
	\end{equation*} where  $\Theta: \mathcal{W}_{\mathcal{H}_W} \to \mathcal{H}_{V} $ is an inner operator and $\mathcal{W}_{\mathcal{H}_W}$ is the generating wandering subspace for $(\pi, W^{(1)}, \dots, W^{(k)}).$
	In this article, we prove the following characterization of the  Beurling quotient subspaces:  A  subspace $\mathcal{K}$ of $\mathcal{ H}_{V}$ is a Beurling quotient subspace  if and only if \begin{align*}&
		(I_{E_{j}}\ot ( (I_{E_{i}}\ot P_{\mathcal{K}})- \wT^{(i) *}\wT^{(i)}))(t_{i,j} \ot I_{\mathcal{H}_{V}})\\&\:\:\:\:\:(I_{E_{i}}\ot ( (I_{E_{j}}\ot P_{\mathcal{K}})- \wT^{(j) *}\wT^{(j)}))=0,\end{align*} where $\widetilde{T}^{(i)}:=P_{\mathcal{K}}\widetilde{V}^{(i)}(I_{E_{i}} \ot P_{\mathcal{K}})$ and $ 1 \leq i,j\leq k.$  As a consequence,  we derive a concrete regular dilation theorem for a pure, completely contractive covariant representation
	$(\sigma, V^{(1)}, \dots, V^{(k)})$ of  $\mathbb{E}$ on a Hilbert space $\mathcal{H}_{V}$ which satisfies Brehmer-Solel condition and using it and the above characterization, we provide a necessary and sufficient condition that when a completely contractive covariant representation is unitarily equivalent to the compression of the induced representation on the Beurling quotient subspace. Further, we study the relation between Sz.Nagy-Foias type factorization of isometric multi-analytic operators and joint invariant subspaces of the compression of the induced representation on the Beurling quotient subspace.
\end{abstract}

%%% ----------------------------------------------------------------------
\maketitle
%%% ----------------------------------------------------------------------
%\tableofcontents
\section{Introduction}
A well known application of Wold decomposition \cite{W} is Beurling's theorem \cite{BA} which is a characterization of the shift invariant subspaces of the Hardy space $H^2(\mathbb D).$ 
Considering doubly commuting isometries, S{\l}oci{\'n}ski
in \cite{Sl80} gave a Wold-type decomposition. In \cite{R74}, Rudin explained that the Beurling theorem fails in general in the multivariable case, that is, for the shift invariant subspaces of the Hardy space on the polydisc. However, assuming $(M_{z_1}, M_{z_2})$ is doubly commuting, Mandrekar \cite{M88} gave a  Beurling's type theorem for the Hardy space over the bidisc $H^2(\mathbb D^2)$ with help of the S{\l}oci{\'n}ski's decomposition. The case of polydisc $H^2(\mathbb D^n),$ was discussed by Sarkar, Sasane and Wick \cite{SSW}. One of the important problem in multivariable operator theory is to find concrete structure of quotient modules (backward shift invariant subspaces) of function spaces, see for example \cite{BC17,CG3,D89,G009,YR19,YZ19}. A definite application of these multivariable Beurling type theorems in the doubly commuting setting is considered recently by	Bhattacharjee, Das, Debnath and  Sarkar in \cite{BDDS}. Indeed, in \cite{BDDS}, an application to the model space theory in terms of regular dilation is discussed based on a characterization of a Beurling quotient module of $H^2(\mathbb D^n).$

Isometric covariant representations of $C^*$-correspondences play an important role in the theory of Cuntz-Pimsner algebras, see \cite{P97}. 
In \cite{MS99}, Muhly and Solel proved Wold decomposition in the case of the isometric covariant representations which is based on Popescu's Wold decomposition \cite{PoB} in the setting of a row
isometry. Tensor product system (see \cite{A89,F02}) is utilized by Solel in \cite{S008} where the notion of doubly commuting covariant representations of the product systems of $C^*$-correspondences is considered to study  regular dilations.
Skalski and Zacharias proved Wold-type decomposition for the doubly commuting isometric covariant representations.  In this direction, Trivedi and  Veerabathiran \cite{HV21}  provided a Beurling–Lax-type theorem for the pure doubly commuting isometric covariant representations of product systems of $C^*$-correspondences which is a generalization  of Beurling–Lax-type theorem due to Popescu \cite{PoB}, Mandrekar \cite{M88}, and Sarkar, Sasane and Wick \cite{SSW}. Along this direction and on the theme of \cite{BDDS}, in this article, we characterize doubly commuting Beurling quotient subspaces and study its application to Sz.Nagy-Foias model space theory in this setting based on regular dilation due to  Solel \cite{S008}  and 	Skalski \cite{S009}.
\subsection{Preliminaries and basic results}
Assume
$\be$ to be a family of $C^*$-correspondences $\{E_1, \ldots, E_k\},$ over a $C^*$-algebra $\mathcal{B}$ along with the unitary
isomorphisms $t_{i,j}: E_i \ot E_j \to E_j \ot E_i$ ($i>j$) where $ 1\leq j<i\leq k $ and  $k\in \mathbb N.$ With the help of these unitary isomorphisms, we may identify $C^*$-correspondence  $\be ({\bf n})$ as $E_1^{\ot^{ n_1}} \ot \cdots \ot E_k^{\ot^{n_k}},$ for all
${\bf n}=(n_1, \cdots, n_k) \in \Nk $ $ (\mathbb{N}_{0}=\mathbb{N}\cup\{0\}).$  Define maps $t_{i,i} := \id_{E_i \ot E_i}$ and $t_{i,j} := t_{j,i}^{-1}$ when $i<j.$ 
In this case, we say that $\be$ is a {\it product system over $\Nk$} (cf. \cite{F02}). 

A {\it completely bounded, covariant representation} (CB-representation) (cf. \cite{S008}) of $\be$ on a
Hilbert space $\mathcal{H}_{V}$ is defined as a tuple $ V:=(\sigma, V^{(1)}, \ldots, V^{(k)})$, where $\sigma$ is a
representation of  $\mathcal{B}$ on $B(\mathcal{ H}_{V}),$ and  $V^{(i)}:E_i \to B(\mathcal{H}_{V})$ are  completely bounded linear  maps satisfying
\[ V^{(i)}(a \xi_i b) = \sigma(a) V^{(i)}(\xi_i) \sigma(b), \;\;\; a,b \in \mathcal B, \xi_i \in E_i,\]
and satisfying the commutation relation
\begin{equation} \label{rep} \wV^{(i)} (I_{E_i} \ot \wV^{(j)}) = \wV^{(j)} (I_{E_j} \ot \wV^{(i)}) (t_{i,j} \ot I_{\mathcal {H}_{V}})\end{equation}
where $1\leq i,j\leq k.$ We shall use notations $\widetilde{V}^{(i)}_l: E_i^{\ot l}\otimes \mathcal {H}_{V}\to \mathcal{ H}_{V}$ defined by
\[ \widetilde{V}^{(i)}_l (\xi_1 \ot \cdots \ot \xi_l \ot h) := V^{(i)} (\xi_1) \cdots V^{(i)}(\xi_l) h\]
where $\xi_1, \ldots, \xi_l \in E_{i}, h \in \mathcal{ H}_{V}$.
%(above we write $\wT^{(i)}$ for $\wt{T_i}$).

We say that two such a completely bounded covariant representations $(\sigma, V^{(1)}, \ldots, V^{(k)})$ and $(\mu, T^{(1)}, \ldots, T^{(k)})$ of the product system $\be$ over $\mathbb{N}^k_0$,
respectively on Hilbert spaces $\mathcal {H}_{V}$ and $\mathcal {H}_{T}$, are {\it isomorphic}  (cf. \cite{SZ08})  if we have a unitary $U:\mathcal {H}_{V} \to
\mathcal {H}_{T}$ which gives the unitary equivalence of representations $\sigma$ and $\psi,$ and also for each $1\leq i \leq k$, $\xi \in E_i$ one has $T^{(i)}(\xi) = U V^{(i)} (\xi) U^*$.

The concept of doubly commuting isometric representation that follows is a generalization of both non-commuting and doubly $\Lambda$-commuting row isometries (see \cite {GP20, MP,BB02,JPS05,Web}).
\begin{definition}  \label{dcom}
	A  CB-representation $(\sigma, V^{(1)}, \ldots,$ $ V^{(k)})$ of  $\be$ on  $\mathcal{H}_{V}$ is said to be {\rm doubly
		commuting isometric representation} (DCI-representation)  (cf. \cite{S008}) if for each distinct $i,j \in \{1,\ldots,k\}$ we have
	\begin{equation*}\label{doubly}\wV^{(j)^*} \wV^{(i)} =
		(I_{E_j} \ot \wV^{(i)})  (t_{i,j} \ot I_{\mathcal{ H}_{V}})  (I_{E_i} \ot \wV^{(j)^*}),
	\end{equation*} and each $\widetilde{V}^{(i)}$ is an isometry. 
\end{definition}
The {\rm Fock module} of  $\mathbb{E},$ is given by
$$\mathcal{F}(\mathbb{E}):=\bigoplus_{\mathbf{n} \in \mathbb{N}^k_0}\mathbb{E}(\mathbf{n}),$$ where the natural left action $\phi_{\infty}$  is given as $\phi_{\infty}(a)(\oplus_{\mathbf{n} \in \mathbb{N}_0^k} \xi_{\mathbf{n}})=\oplus_{\mathbf{n} \in \mathbb{N}_0^k} a \cdot \xi_{\mathbf{n}},\: \xi_{\mathbf{n}} \in \mathbb{E}(\mathbf{n}), a \in \mathcal{B}.$
Define for $\mathbf{m}=(m_1, \cdots, m_k) \in \mathbb{N}_0^k $, the map $V_{\mathbf{m}}: \mathbb{E}(\mathbf{m}) \to B(\mathcal{H}_{V})$  by $V_{\mathbf{m}}(\xi_{\mathbf{m}})h:=\widetilde{V}_{\mathbf{m}}(\xi_{\mathbf{m}} \otimes h), ~ \xi_{\mathbf{m}} \in \mathbb{E}(\mathbf{m}), h \in \mathcal{H}_{V},$ where $\widetilde{V}_{\mathbf{m}}:\mathbb{E}(\mathbf{m})\otimes_{\sigma} \mathcal{H}_{V} \to \mathcal{H}_{V}$ is defined by
$$\widetilde{V}_{\mathbf{m}}=\widetilde{V}^{(1)}_{m_1}\left(I_{E_1^{\otimes m_1}} \otimes\widetilde{V}^{(2)}_{m_2}\right) \cdots \left(I_{E_1^{\otimes m_1} \otimes \cdots \otimes E_{k-1}^{\otimes {m_{k-1}}}} \otimes\widetilde{V}^{(k)}_{m_k}\right).$$ 

Now we recall the definition of invariant subspace (cf. \cite{SZ08}) for the covariant representation $(\sigma, V^{(1)}, \dots, V^{(k)})$ of $\mathbb{E}$ on  $\mathcal{H}_{V}.$ Throughout this article, we use the notation $I_{k}$ for $\{1,2,\dots, k\},$ where $ k \in \mathbb{N}.$
\begin{definition}	 
	\begin{itemize}
		\item[$(1)$] Consider a CB-representation $(\sigma, V^{(1)}, \dots, V^{(k)})$ of $\mathbb{E}$ on  $\mathcal{H}_{V}$ and suppose  $\mathcal K$ is a closed subspace of $ \mathcal{ H}_{V}.$  Then we say $\mathcal K$ is  $(\sigma, V^{(1)}, \dots, V^{(k)})$-{\rm invariant}   if it  is $\sigma(\mathcal B)$-invariant
		%(i.e., the projection onto $\mathcal K$, will be denoted throughout by $P_{\mathcal K}$, lies  in
		%$\sigma(\mathcal B)')$,
		and, is invariant by each operator $V^{(i)}(\xi_{i}),\: \xi_{i} \in E_{i}, i \in I_{k}.$  In addition, if  $ \mathcal{K}^{\bot}$ is   invariant by $V^{(i)}(\xi_{i})$ for $\xi_{i} \in E_{i}, i \in I_{k}$ then we say $\mathcal{K}$ is $(\sigma, V^{(1)}, \dots, V^{(k)})$-{\em reducing}.  Restricting naturally  this representation we get  another representation of $E$ on $\mathcal{K}$ which will be denoted as $(\sigma, V^{(1)}, \dots, V^{(k)})|_{\mathcal{K}}.$ 
		
		\item[$(2)$]  A closed  subspace $\mathcal{W}$ of $\mathcal{H}_{V}$ is called {\rm wandering} subspace  for\\ $(\sigma, V^{(1)}, \dots, V^{(k)})$, if it is $\sigma(\mathcal{B})$-invariant and $\mathcal{W}\perp\wt{V}_{\mathbf n}(\mathbb{E}({\mathbf n} )\ot \mathcal{W})$ for every $n \in \mathbb{N}_{0}^{k}.$ The representation $(\sigma, V^{(1)}, \dots, V^{(k)})$ has {\rm generating wandering subspace property} (GWS-property) if there is a wandering subspace $\mathcal{W}$ of $\mathcal{H}_{V}$ satisfying
		$$\mathcal{H}_{V}=\displaystyle\bigvee_{\mathbf n\in \mathbb{N}_0^{k}}\wt{V}_{\mathbf n}(\mathbb{E}({\mathbf  n}) \ot \mathcal{W})$$
		and the corresponding  wandering subspace $\mathcal{W}$
		is called generating wandering subspace (GWS). 
	\end{itemize}
	
\end{definition}

Consider  pure DCI-representations $(\sigma, V^{(1)}, \dots, V^{(k)})$ and $(\mu, T^{(1)}, \dots, T^{(k)})$ of $\mathbb{E}$ on the Hilbert spaces $\mathcal{H}_{V}$ and $\mathcal{H}_{T},$ respectively. Then using \cite[Corollary 3.5]{HV21}, we get $$\mathcal{H}_{V}=\bigoplus_{\mathbf{n} \in \mathbb{N}_0^k}\wt{V}_{\mathbf{n}}(\mathbb{E}(\mathbf{n}) \otimes \mathcal{W}_{\mathcal{H}_{V}} ) \:\: \mbox{and} \:\: \mathcal{H}_{T}=\bigoplus_{\mathbf{n} \in \mathbb{N}_0^k}\wt{T}_{\mathbf{n}}(\mathbb{E}(\mathbf{n}) \otimes \mathcal{W}_{\mathcal{H}_{T}} ),$$ where $\mathcal{W}_{\mathcal{H}_{V}}$ and $\mathcal{W}_{\mathcal{H}_{T}}$ are the generating wandering subspaces for $(\sigma, V^{(1)}, \dots, V^{(k)})$ and $(\mu, T^{(1)}, \dots, T^{(k)}),$ respectively.

A bounded operator $A: \mathcal{H}_{V}\to \mathcal{H}_{T}$ is called {\it multi-analytic} (cf. \cite{HV21}) if it satisfies 
\begin{align*}
	AV^{(i)}(\xi_i)h=T^{(i)}(\xi_i)Ah \hspace{0.7cm}\mbox{and} \hspace{0.7cm}A \sigma(a)h= \mu(a)Ah,
\end{align*}
where $\xi_i \in E_i, \: h \in \mathcal{H}_{V},$ $ a \in \mathcal{B}$ and $ i \in I_{k}.$ Each such $A$ is uniquely determined by  $\Theta: \mathcal{W}_{\mathcal{H}_{V}} \to \mathcal{H}_{T}$ satisfying  $\Theta \sigma(a)h=\mu(a)\Theta h, h \in \mathcal{W}_{\mathcal{H}_{V}},$ where $ \Theta= A|_{\mathcal{W}_{\mathcal{H}_{V}}}.$ This follows  because for every $\xi_{\mathbf{n}} \in \mathbb{E}(\mathbf{n}), h \in \mathcal{W}_{\mathcal{H}_{V}}$ we have $AV_{\mathbf{n}}(\xi_{\mathbf{n}})h=T_{\mathbf{n}}(\xi_{\mathbf{n}})\Theta h$ and $\mathcal{H}_{V}=\bigoplus_{\mathbf{n} \in \mathbb{N}_0^k}\wt{V}_{\mathbf{n}}(\mathbb{E}(\mathbf{n}) \otimes \mathcal{W}_{\mathcal{H}_{V}} ).$ Conversely, if we start with $\Theta: \mathcal{W}_{\mathcal{H}_{V}} \to \mathcal{H}_{T}\\\left( = \bigoplus_{\mathbf{n} \in \mathbb{N}_0^k}\wt{T}_{\mathbf{n}}(\mathbb{E}(\mathbf{n}) \otimes \mathcal{W}_{\mathcal{H}_{T}} ) \right)$ which  satisfies $\Theta \sigma(a)h=\mu(a)\Theta h, h \in \mathcal{W}_{\mathcal{H}_{V}},$ then we get $M_{\Theta}: \mathcal{H}_{V} \to \mathcal{H}_{T}$ defined by the equation $$M_{\Theta}V_{\mathbf{n}}(\xi_{\mathbf{n}})h=T_{\mathbf{n}}(\xi_{\mathbf{n}})\Theta h=T_{\mathbf{n}}(\xi_{\mathbf{n}})M_{\Theta} h, \:\:\: \xi_{\mathbf{n}} \in \mathbb{E}(\mathbf{n}), h \in \mathcal{W}_{\mathcal{H}_{V}},$$ which is multi-analytic.
We shall always consider $\Theta$ such that $M_{\Theta}$ is a contraction. It is easy to check that
$$M_{\Theta}\left(\bigoplus_{\mathbf{n} \in \mathbb{N}^k_0}h_{\mathbf{n}} \right)=\sum_{\mathbf{n} \in \mathbb{N}^k_0} \wt{T}_{\mathbf{n}}(I_{\mathbb{E}(\mathbf{n})} \otimes \Theta)\wt{V}_{\mathbf{n}}^*h_{\mathbf{n}} \hspace{1cm} \mbox{for}\: \bigoplus_{\mathbf{n} \in \mathbb{N}^k_0}h_{\mathbf{n}} \in \mathcal{H}_{V}.$$

The map $\Theta,$ as defined above, will be called (cf. \cite[Proposition 4.4]{HV21}) 
\begin{enumerate}
	\item[$(1)$]{\it outer} if $\overline{M_{\Theta}\mathcal{H}_{V}}=\mathcal{H}_{T}$ (equivalently, $\Theta \mathcal{W}_{\mathcal{H}_{V}} $ is cyclic for  $T,$ i.e., $$\bigvee_{\mathbf{n} \in \mathbb{N}^k_0}\wt{T}_{\mathbf{n}}(\mathbb{E}(\mathbf{n}) \otimes  \Theta \mathcal{W}_{\mathcal{H}_{V}} )=\mathcal{H}_{T}),$$ 
	\item[$(2)$] {\it inner} if $M_{\Theta}$ is an isometry, (equivalently, $\Theta \mathcal{W}_{\mathcal{H}_{V}} $ is a wandering subspace for $T$  and $\Theta$ is an isometry).
\end{enumerate}
Indeed, it follows that $\Theta$ is inner and  outer if and only if $\Theta$ is a unitary operator from $ \mathcal{W}_{\mathcal{H}_{V}}$ to $\mathcal{W}_{\mathcal{H}_{T}}.$

\begin{definition}
	Let $(\sigma, V^{(1)}, \dots, V^{(k)})$  be a pure isometric  representation of  $\mathbb{E}$ on $\mathcal{H}_{V}$  and $\mathcal{K}$ be a $\sigma$-invariant subspace (IS) of $\mathcal{H}_{V}$, i.e., $\sigma (b)\mathcal{ K}\subseteq\mathcal{ K}$ for each $b \in \mathcal{B}.$ Then  $\mathcal{K}$ is said to be  {\rm quotient subspace} (QS) of $\mathcal{H}_{V}$ if for each $ i\in I_{k}$, $ 
	\widetilde{V}^{(i)*}\mathcal{K}\subseteq E_{i}\ot \mathcal{K}.
	$\end{definition}
\begin{remark}
	Note that $\mathcal{K}$ is $(\sigma, V^{(1)}, \dots, V^{(k)})$-IS  of $\mathcal{H}_{V}$  if and only if  $\mathcal{K}^{\perp}$  is a QS of $\mathcal{H}_{V}. $
\end{remark}
The following  Beurling-Lax-Halmos type theorem \cite[Theorem 4.4]{HV21} provides a characterization of  QS $\mathcal{ K}$ of $\mathcal{H} _V$, which is a generalization of Popescu's version of Beurling-Lax theorem \cite[Theorem 2.2]{G89}. 
\begin{theorem}\label{hs} Let $(\sigma, V)$ be a pure isometric representation of  ${E}$ on $\mathcal{H}_{V}$ and let $\mathcal{K}$ be a  closed subspace of $\mathcal{ H}_{V}$. Then $\mathcal{K}$ is a QS of $\mathcal{ H}_{V}$ if and only if there exist a Hilbert space $\mathcal{H}_{T},$ a  pure isometric representation $(\mu, T)$ of $E$ on $\mathcal{H}_{T}$ and an isometric multi-analytic operator $M_\Theta: \mathcal{H}_{T} \to \mathcal{H}_{V}$ such that $$\mathcal{K}=\mathcal{H}_{V}\ominus M_{\Theta}\mathcal{H}_{T},$$ where $\Theta :\mathcal{W}_{\mathcal{H}_{T}} \to \mathcal{H}_{V}  $ is inner operator and  $\mathcal{W}_{\mathcal{H}_{T}}$ is the generating wandering subspace for $(\mu, T).$
\end{theorem}

Generally, a quotient subspace  $\mathcal{K}$ of $\mathcal{H}_{V}$ for the pure DCI-representation $(\sigma, V^{(1)}, \dots, V^{(k)})$ does not necessarily have the Beurling-type representation \cite{D89, WR, YR19}. In this article, we discuss the Beurling-type representations of quotient subspaces for  a  pure DCI-representation $(\sigma, V^{(1)}, \dots, V^{(k)})$ based on \cite{BDDS}. On the basis \cite[Theorem 4.11]{HV21}, we define a Beurling quotient subspace  for a pure DCI-representation $(\sigma, V^{(1)}, \dots, V^{(k)})$ of  $\mathbb{E}$ on a Hilbert space $\mathcal{H}_{V},$ that is,  a  $\sigma$-IS $\mathcal{K}$  is said to be {\it Beurling quotient subspace} (BQS) of $\mathcal{H}_{V}$ if there exist a Hilbert space $\mathcal{H}_W,$ a pure DCI-representation  $(\pi, W^{(1)}, \dots, W^{(k)})$ of $\mathbb{E}$ on  $\mathcal{H}_W$ and an isometric multi-analytic operator $M_\Theta:{\mathcal{H}_W} \to \mathcal{H}_{V}$ such that
\begin{equation*}
	\mathcal{K}=\mathcal{H}_{V}\ominus M_{\Theta}\mathcal{H}_W,
\end{equation*} where $\mathcal{W}_{\mathcal{H}_W}$ is the generating wandering subspace for $(\pi, W^{(1)}, \dots, W^{(k)}).$ The following natural question arises: 

\begin{question}\label{ques}
	Suppose $(\sigma, V^{(1)}, \dots, V^{(k)})$ is a  pure DCI-representation of  $\mathbb{E}$ on $\mathcal{H}_{V}.$  Which QS of $\mathcal{H}_{V}$ admits Beurling-type representation?
\end{question} 
To answer this question, consider a  pure DCI-representation\\  $(\sigma, V^{(1)}, \dots, V^{(k)})$ of  $\mathbb{E}$ on  $\mathcal{H}_{V}$ and let  $\mathcal{K}$ be a QS of $\mathcal{H}_{V}.$ Define  $$\mu(a):=P_{\mathcal{ K}}\sigma(a){P_{\mathcal{K}}}  \:\mbox{and} \:\:T^{(i)}(\xi_{i}):=P_{\mathcal{K}}V^{(i)}(\xi_{i}){P_{\mathcal{K}}},$$
for $\xi_{i} \in E_{i}$ and $a \in \mathcal{B},$ i.e.,
$\wT^{(i)}:=P_{\mathcal{K}}\widetilde{V}^{(i)}(I_{E_{i}}\ot P_{\mathcal{K}}),$ where $P_{\mathcal{K}}$ is orthogonal projection of $\mathcal{H}_{V}$ on $\mathcal{ K}$,  $i\in I_{k}$  and hence $(\mu,T^{(1)},T^{(2)}, \dots, T^{(k)}) $ is a contractive covariant representation of $\mathbb{E}$ on $\mathcal{K}.$ The following theorem based on \cite[Theorem 1.1]{BDDS} answers  Question \ref{ques} which is the main result of Section \ref{2}.
\begin{theorem}\label{Beurlin}
	Let $(\sigma, V^{(1)}, \dots, V^{(k)})$  be a pure DCI-representation of  $\mathbb{E}$ on  $\mathcal{H}_{V}.$ Let $\mathcal{K}$ be a QS of $\mathcal{H}_{V}$. Then $\mathcal{K}$ is a BQS of $\mathcal{H}_{V}$ if and only if \begin{equation*}
		(I_{E_{j}}\ot ( (I_{E_{i}}\ot P_{\mathcal{K}})- \wT^{(i) *}\wT^{(i)}))(t_{i,j} \ot I_{\mathcal{K}})(I_{E_{i}}\ot ( (I_{E_{j}}\ot P_{\mathcal{K}})- \wT^{(j) *}\wT^{(j)}))=0.
	\end{equation*}
\end{theorem}

In Section \ref{3}, first, we give a concrete regular dilation theorem for a pure completely contractive covariant representation $(\sigma, V^{(1)}, V^{(2)},\dots, V^{(k)})$ of $\mathbb{E}$ on $\mathcal{H}_{V}$ satisfying Brehmer-Solel condition. For $k=2,$  $(\sigma, V^{(1)}, V^{(2)})$   satisfies {\it Brehmer-Solel condition} \cite{BS,S008} if \begin{equation*}
	\Delta_{*}(V):=I_{\mathcal{H}}-\widetilde{V}^{(1)}\widetilde{V}^{(1)*}-\widetilde{V}^{(2)}\widetilde{V}^{(2)*}+\widetilde{V}^{(1)}(I_{E_1} \ot \widetilde{V}^{(2)}\widetilde{V}^{(2)*})\widetilde{V}^{(1)*}\geq 0.
\end{equation*}Using the  concrete dilation theorem there exists a QS $\mathcal{ K}$ of $\mathcal{F}(\mathbb{E}) \ot \mathcal{ H}$ such that   \begin{equation}\label{imp}
	(\sigma, V^{(1)}, V^{(2)})\cong(\rho|_{\mathcal{K}}, S^{(1)}|_{\mathcal{K}}, S^{(2)}|_{\mathcal{K}}),
\end{equation}  where $(\rho, S^{(1)}, S^{(2)})$ is the induced  representation of $\mathbb{E}$ on $\mathcal{F}(\mathbb{E}) \ot \mathcal{ H}.$ But it is not necessary that $\mathcal{K}$ obtained in this way is a BQS of $\mathcal{F}(\be) \ot \mathcal{H}.$ This question is further explored in  this Section \ref{3} with the help of Theorem \ref{Beurlin}.

Sz. Nagy-Foias \cite[Chapter VII, Theorem 2.3]{NFS} provided a necessary and sufficient for a nontrivial invariant subspace of the completely non-unitary contraction $V$ on a Hilbert space if the characteristic function of $V$ admits a nontrivial factorization. Bercovici \cite[Chapter 5, Proposition 1.21]{B88} extended this result and described the invariant subspaces of a functional model. In Section \ref{4} we prove the following result which is  a generalization of this result. \begin{theorem}
	Let $(\sigma, V)$ and $(\mu, T)$ be pure isometric  representations of $E$ on Hilbert spaces $\mathcal{H}_{V}$ and $\mathcal{H}_{T},$ respectively.  Suppose ${\mathcal{W}_{\mathcal{H}_{V}}}$  and ${\mathcal{W}_{\mathcal{H}_{T}}}$ are the generating wandering subspace for $(\sigma, V)$ and $(\mu,T),$ respectively. Let ${M}_\Theta:\mathcal{H}_{V}\to \mathcal{H}_{T}$ be an isometric multi-analytic operator. Then  the covariant representation $(\pi'_{\Theta}, W'_{\Theta})$ has an IS   if and only if there exists a Hilbert space $\mathcal{H}_{R},$ a pure isometric  representation $(\nu, R)$ of $E$ on $\mathcal{H}_{R}$, and isometric multi-analytic operators ${M}_\Phi: {\mathcal{H}_{R}} \to {\mathcal{H}_{T}}$, ${M}_\Psi: {\mathcal{H}_{V}} \to {\mathcal{H}_{R}}$ such that  \begin{equation*}
		{\Theta}={\Phi}{\Psi}.
	\end{equation*}
\end{theorem}
Moreover, if we take $k=2,$ i.e., $(\pi'_{\Theta}, W_{\Theta}^{'(1)}, W_{\Theta}^{'(2)})$-joint IS alone does not guarantee the factorization of an inner operator, and this is a motivation for Section \ref{4}. 
More generally based on \cite[Theorem 4.4]{BDDS}  we prove that if $(\sigma, V^{(1)}, \dots, V^{(k)})$ and $(\mu, T^{(1)}, \dots, T^{(k)})$  be DCI-representations of  $\mathbb{E}$ on the Hilbert spaces $\mathcal{H}_{V}$ and $\mathcal{H}_{T},$ respectively and  ${M}_\Theta:\mathcal{H}_{V}\to \mathcal{H}_{T}$ be an isometric multi-analytic operator, then  $\Theta$ factorize  in terms of inner operators.

%Similarly for $\mathbf{n}=(n_1, \cdots, n_k) \in \mathbb{N}_0^k $, we use notation $\wT_{\mathbf{n}}:\mathbb{E}(\mathbf{n})\otimes \mathcal{H} \to \mathcal{H}$ for
% $$\wT_{\mathbf{n}}:=\wT^{(1)}_{n_1}\left(I_{E_1^{\otimes n_1}} \otimes\wT^{(2)}_{n_2}\right) \cdots \left(I_{E_1^{\otimes n_1} \otimes \cdots \otimes E_{k-1}^{\otimes {n_{k-1}}}} \otimes\wT^{(k)}_{n_k}\right).$$

%Let us define the linear map $T_{\mathbf{n}}: \mathbb{E}(\mathbf{n}) \to B(\mathcal{H})$  (cf. \cite{S08}) by $$T_{\mathbf{n}}(\xi)h:=\wT_{\mathbf{n}}(\xi \otimes h), ~ \xi \in \mathbb{E}(\mathbf{n}), h \in \mathcal{H}.$$
%We use  $I_k$ for $I_{k}.$\section{Beurling Quotient Subspace For The Covariant Representations}\label{2}

\section{Beurling quotient subspace for the covariant representations}\label{2}

In this section, we prove a necessary and sufficient condition regarding when a QS becomes  BQS for DCI-representations of a product system. For this purpose, we need some fundamental lemmas regarding BQS and cross-commutators. First, we recall a Beurling-type characterization for the doubly commuting invariant subspaces for DCI- representations.

Let   $(\sigma, V^{(1)}, \dots, V^{(k)})$   be a covariant representation  of  $\mathbb{E}$ on a Hilbert space  $\mathcal{H}_{V}$ and $\mathcal{K}$ be a $(\sigma, V^{(1)}, \dots, V^{(k)})$-invariant subspace of $\mathcal{H}_{V}.$ The natural restriction of  $(\sigma, V^{(1)}, \dots, V^{(k)})$  to  $\mathcal{K}$  provides a new representation of $\mathbb{E}$ on $\mathcal{K}$ and it will be denoted by $(\sigma, V^{(1)}, \dots, V^{(k)})|_{\mathcal{K}}.$ This means that, for  $i \in I_k$ define $ V|_{\mathcal{K}}^{(i)}: E_i   \rightarrow B(\mathcal{K})$ by $V|_{\mathcal{K}}^{(i)}(\xi_i)h=V^{(i)}(\xi_i)h, \xi_i \in E_i, h \in \mathcal{K},$ and  the restriction map    ${\wt{V}|_{\mathcal{K}}^{(i)}}: E_{i}\ot \mathcal{ K} \rightarrow  \mathcal{K}$ by ${\wt{V}|_{\mathcal{K}}^{(i)}}(\xi_i\ot h)=V^{(i)}(\xi_i \ot h).$ That is,  $(\sigma, V^{(1)}, \dots, V^{(k)})|_{\mathcal{K}}=(\sigma|_{\mathcal{K}}, V|_{\mathcal{K}}^{(1)}, \dots, V|_{\mathcal{K}}^{(k)}).$

%  ${\wt{V}|_{\mathcal{K}}^{(i)}}$ is the restriction map to $E_{i}\ot \mathcal{ K},$ that is, ${\wt{V}|_{\mathcal{K}}^{(i)}}:{E_{i}\ot \mathcal{ K}} \to \mathcal{ K}$ by ${\wt{V}|_{\mathcal{K}}^{(i)}}(\xi_i\ot h)=V^{(i)}(\xi_i \ot h).$ 
\begin{thm}\cite[Theorem 4.11]{HV21}\label{MT5}
	Let $(\sigma, V^{(1)}, \dots, V^{(k)})$  be a pure DCI-representation of  $\mathbb{E}$ on the Hilbert space  $\mathcal{H}_{V}.$  Let $\mathcal{K}$ be a  $(\sigma, V^{(1)}, \dots, V^{(k)})$-IS of $\mathcal{H}_{V}$. Then  $\mathcal{K}$ is a $(\sigma, V^{(1)}, \dots, V^{(k)})$-{\rm doubly commuting subspace (DCS)}, that is,
	$${\wt{V}|_{\mathcal{K}}^{(i)}}^{*}{\wt{V}|_{\mathcal{K}}^{(j)}}=(I_{E_i} \otimes {\wt{V}|_{\mathcal{K}}^{(j)}})(t_{j,i} \otimes I_{\mathcal{K}})(I_{E_j} \otimes{\wt{V}|_{\mathcal{K}}^{(i)}}^{*}),$$  where $i,j \in I_k$ with $ i \neq j,$  if and only if there exist a Hilbert space $\mathcal{H}_{T}$, a pure DCI-representation $(\mu, T^{(1)}, \dots, T^{(k)})$ of $\mathbb{E}$ on  $\mathcal{H}_T$ and an isometric multi-analytic operator $M_\Theta:{\mathcal{H}_T} \to \mathcal{H}_{V}$ such that
	$$\mathcal{K}=M_{\Theta}\mathcal{H}_T.$$ $\mathcal{ K}$ is also called as {\it Beurling subspace}. Indeed if $\mathcal{K}=\mathcal{H}_{V},$ then $\Theta$ is outer.
\end{thm}
Note that any reducing subspace for the covariant representation \\ $(\sigma, V^{(1)}, \dots, V^{(k)})$  is doubly commuting subspace. However, the converse need not be true in general. 

\begin{remark}\label{1} Note that  if $\mathcal{K}$ is a $(\sigma, V^{(1)}, \dots, V^{(k)})$-IS of $\mathcal{ H}_{V},$ then for each $ i \in I_{k},$ ${\wt{V}|_{\mathcal{K}}^{(i)}}=P_{\mathcal{ K}}{\wt{V}|_{\mathcal{K}}^{(i)}}$ and if $\mathcal{K}$ is a $(\sigma, V^{(1)}, \dots, V^{(k)})$-co-invariant subspace of $\mathcal{ H}_{V},$ then ${\wt{V}|_{\mathcal{K}^\perp}^{(i)}}^{*}=P_{{E_{i}\ot \mathcal{K}^{\perp}}}\widetilde{V}^{(i)*}|_{\mathcal{ K}^{\perp}}.$ 
\end{remark}
Consider a completely bounded representation $(\sigma, V^{(1)}, \dots, V^{(k)})$    of  $\mathbb{E}$ on  $\mathcal{H}_{V}$ and let $\mathcal{K}$ be a subspace of $\mathcal{ H}_{V}.$  
For each $i,j \in I_{k},$  we denote the cross commutator of ${\wt{V}|_{\mathcal{K}}^{(j)}}^{*}$ and ${\wt{V}|_{\mathcal{K}}^{(i)}}$ by $[{\wt{V}|_{\mathcal{K}}^{(j)}}^{*}, {\wt{V}|_{\mathcal{K}}^{(i)}}]$ and  define it by \begin{align*}
	&[{\wt{V}|_{\mathcal{K}}^{(j)}}^{*}, {\wt{V}|_{\mathcal{K}}^{(i)}}]:=	{\wt{V}|_{\mathcal{K}}^{(j)}}^{*}{\wt{V}|_{\mathcal{K}}^{(i)}} - (I_{E_{j}} \ot {\wt{V}|_{\mathcal{K}}^{(i)}})(t_{i,j} \ot I_{\mathcal{K}})(I_{E_{i}}\ot {\wt{V}|_{\mathcal{K}}^{(j)}}^{*}).
\end{align*} 
\begin{remark}\label{11}
	It follows from Theorem \ref{MT5}  that  $\mathcal{K}^{\perp}$ is a  $(\sigma, V^{(1)}, \dots, V^{(k)})$-DCS of $\mathcal{H}_{V}$ if and only if \begin{align*}
		[\widetilde{R}^{({j})*}, \widetilde{R}^{(i)}]=0,\:\: \mbox{for all}\:\: i\neq j,
	\end{align*} where $\widetilde{R}^{(i)}:={\wt{V}|_{\mathcal{K}^{\perp}}^{(i)}} , \:\: i,j \in I_{k}.$ In this case, we say that $\mathcal{ K}$ is a BQS of $\mathcal{H}_{V}.$ This gives the following lemma.
\end{remark}

\begin{lemma}\label{Lemma 2.1}	Let $(\sigma, V^{(1)}, \dots, V^{(k)})$  be a pure DCI-representation of  $\mathbb{E}$ on  $\mathcal{H}_{V}$ and let $\mathcal{K}$ be a QS of $\mathcal{H}_{V}.$ For each $i, j \in I_{k}$, define 
	\begin{equation*}
		X_{i,j}:=(I_{E_{j}} \ot P_{ \mathcal{K}^{\perp}} \widetilde{V}^{(i)})(t_{i,j} \ot P_{\mathcal{K}})(I_{E_{i}}\ot \widetilde{V}^{(j)*} P_{\mathcal{K}^\perp}).
	\end{equation*}
	Then $\mathcal{K}$ is a BQS of $\mathcal{H}_{V}$ if and only if $X_{i,j}=0$ for  $i \neq j$.
\end{lemma}	
\begin{proof}
	For distinct $i, j\in I_{k},$ using the DCS condition and Remark \ref{1}
	we obtain  \begin{align*}[\widetilde{R}^{({j})*}, \widetilde{R}^{(i)}]&=[{\wt{V}|_{\mathcal{K}^{\perp}}^{(j)}}^{*}, {\wt{V}|_{\mathcal{K}^{\perp}}^{(i)}}]=	{\wt{V}|_{\mathcal{K}^{\perp}}^{(j)}}^{*}{\wt{V}|_{\mathcal{K}^{\perp}}^{(i)}} \\&\:\:\:\:\: - (I_{E_{j}} \ot {\wt{V}|_{\mathcal{K}^{\perp}}^{(i)}})(t_{i,j} \ot I_{\mathcal{K}^{\perp}})(I_{E_{i}}\ot {\wt{V}|_{\mathcal{K}^{\perp}}^{(j)}}^{*})\\&=P_{{E_{j}\ot \mathcal{K}^{\perp}}}\widetilde{V}^{(j)*}{\wt{V}|_{\mathcal{K}^{\perp}}^{(i)}}-P_{{E_{j}\ot \mathcal{K}^{\perp}}}(I_{E_{j}} \ot {\wt{V}^{(i)}})(t_{i,j} \ot P_{\mathcal{K}^{\perp}})(I_{E_{i}}\ot {\wt{V}|_{\mathcal{K}^{\perp}}^{(j)}}^{*})\\&=P_{{E_{j}\ot \mathcal{K}^{\perp}}}(I_{E_{j}} \ot {\wt{V}^{(i)}})(t_{i,j} \ot P_{\mathcal{K}})(I_{E_{i}}\ot {\wt{V}|_{\mathcal{K}^{\perp}}^{(j)}}^{*}),
	\end{align*}  where $P_{\mathcal{K}^{\perp}}$ is an orthogonal projection of $\mathcal{H}_{V}$ on $\mathcal{K}^{\perp}$ and $P_{E_{j} \ot \mathcal{K}^{\perp}}$ is an orthogonal projection of $E_{j} \ot \mathcal{H}_{V}$ on $E_{j} \ot \mathcal{K}^{\perp}.$  Hence Remark \ref{11} gives that $\mathcal{ K}^{\perp}$ is a DCS of $\mathcal{ H}_{V}$  if and only if $${\big((I_{E_{j}} \ot P_{ \mathcal{K}^{\perp}})(I_{E_{j}}\ot \widetilde{V}^{(i)})(t_{i,j} \ot P_{\mathcal{K}})(I_{E_{i}}\ot \widetilde{V}^{(j)*})\big)}|_{E_{i} \ot \mathcal{K}^{\perp}}=0,$$ which means  $X_{i,j}=0.$\end{proof}

Consider a DCI-representation $(\sigma, V^{(1)}, \dots, V^{(k)})$ of  $\mathbb{E}$ on  $\mathcal{H}_{V}$ and let $\mathcal{K}$ be a QS of $\mathcal{H}_{V}.$ Define \begin{align*}
	\mu(a):=P_{\mathcal{K}}\sigma(a)P_{\mathcal{K}}  \:\mbox{and} \:\:T^{(i)}(\xi_{i}):=P_{\mathcal{K}}V^{(i)}(\xi_{i})P_{\mathcal{K}}, \:\:\: \text {for each}\:\: \xi_{i} \in E_{i},a \in \mathcal{B}, 
\end{align*}
i.e.,
$\wT^{(i)}=P_{\mathcal{K}}\widetilde{V}^{(i)}(I_{E_{i}}\ot P_{\mathcal{K}}),\:\:\:\:i\in I_{k}.$ Therefore $T=(\mu,T^{(1)},T^{(2)}, \dots, T^{(k)}) $ is a completely contractive covariant representation of $\mathbb{E}$ on $\mathcal{K}$  having the commutant relation (\ref{rep}).

Throughout in this section we fix $i,j \in I_{k}$ and assume that $i\neq j$ and also $ 
\mathbf{\hat{m}}_{i}$ denote multi-indices in $\mathbb{N}^{k}_{0}$ whose  $i^{th}$ entry is zero. The following lemma establishes a relationship between a DCI-representation and its compression on $\mathcal{ K}$.
\begin{lemma}\label{2.2}
	Let $(\sigma, V^{(1)}, \dots, V^{(k)})$  be a DCI-representation of  $\mathbb{E}$ on  $\mathcal{H}_{V}.$ Let $\mathcal{K}$ be a QS of $\mathcal{H}_{V}$. Then \begin{equation*}
		[\wT^{(i)}, \wT_{\mathbf{\hat{m}}_{i}}^{*}]=(I_{\mathbb{E}{(\mathbf{\hat{m}}_{i}})} \ot P_{\mathcal{K}})\widetilde{V}_{\mathbf{\hat{m}}_{i}}^{*}P_{\mathcal{K}^{\perp}}\widetilde{V}^{(i)}(I_{E_{i}}\ot P_{\mathcal{K}}).
	\end{equation*}
\end{lemma}
\begin{proof}
	Observe that $\widetilde{T}^{*}_{\mathbf{m}}=\widetilde{V}_{\mathbf{m}}^{*}P_{\mathcal{K}}$ and therefore $\widetilde{T}_{\mathbf{m}}=P_{\mathcal{K}}\widetilde{V}_{\mathbf{m}}$   for all $\mathbf{m}=(m_1,m_2,\dots,m_k)\in \mathbb{N}^k_{0}.$  Substituting the values of  $\wT_{\mathbf{\hat{m}}_{i}}^{*}$ and $\wT^{(i)},$  we get\begin{align*}
		&	[\wT^{(i)}, \wT_{\mathbf{\hat{m}}_i}^{*}]\\&=(I_{\mathbb{E}{(\mathbf{\hat{m}}_{i}})} \ot P_{\mathcal{K}}\widetilde{V}^{(i)}(I_{E_{i}}\ot P_{\mathcal{K}}))(t_{i,\mathbf{\hat{m}}_{i}}\ot I_{\mathcal{H}_{V}})(I_{E_{i}}\ot\widetilde{V}_{\mathbf{\hat{m}}_{i}}^{*}P_{\mathcal{K}})\\&\:\:\:\:\:\:\:-\widetilde{V}_{\mathbf{\hat{m}}_{i}}^{*}P_{\mathcal{K}}\widetilde{V}^{(i)}(I_{E_{i}}\ot P_{\mathcal{K}})\\&=(I_{\mathbb{E}{(\mathbf{\hat{m}}_{i}})} \ot P_{\mathcal{K}} \widetilde{V}^{(i)})(t_{i,\mathbf{\hat{m}}_{i}}\ot I_{\mathcal{H}_{V}})( I_{E_{i}}\ot ((I_{\mathbb{E}{(\mathbf{\hat{m}}_{i})}}\ot P_{\mathcal{K}})\widetilde{V}_{\mathbf{\hat{m}}_{i}}^{*} P_{\mathcal{K}}))\\&\:\:\:\:\:\:\:-\widetilde{V}_{\mathbf{\hat{m}}_{i}}^{*}P_{\mathcal{K}}\widetilde{V}^{(i)}(I_{E_{i}}\ot P_{\mathcal{K}})\\&=(I_{\mathbb{E}{(\mathbf{\hat{m}}_{i}})} \ot P_{\mathcal{K}} \widetilde{V}^{(i)})(t_{i,\bf{\hat{m}}_{i}}\ot I_{\mathcal{H}_{V}})( I_{E_{i}}\ot \widetilde{V}_{\mathbf{\hat{m}}_{i}}^{*} P_{\mathcal{K}})-\widetilde{V}_{\mathbf{\hat{m}}_{i}}^{*}P_{\mathcal{K}}\widetilde{V}^{(i)}(I_{E_{i}}\ot P_{\mathcal{K}})\\&=(I_{\mathbb{E}{(\mathbf{\hat{m}}_{i}})} \ot P_{\mathcal{K}})\widetilde{V}_{\mathbf{\hat{m}}_{i}}^{*} \widetilde{V}^{(i)}( I_{E_{i}}\ot P_{\mathcal{K}})-(I_{\mathbb{E}{(\mathbf{\hat{m}}_{i}})} \ot P_{\mathcal{K}})\widetilde{V}_{\mathbf{\hat{m}}_{i}}^{*}P_{\mathcal{K}}\widetilde{V}^{(i)}(I_{E_{i}}\ot P_{\mathcal{K}})\\&=(I_{\mathbb{E}{(\mathbf{\hat{m}}_{i}})} \ot P_{\mathcal{K}})\widetilde{V}_{\mathbf{\hat{m}}_{i}}^{*}[I_{\mathcal{H}_{V}}-P_{\mathcal{K}}]\widetilde{V}^{(i)}(I_{E_{i}}\ot P_{\mathcal{K}})\\&=(I_{\mathbb{E}{(\mathbf{\hat{m}}_{i}})} \ot P_{\mathcal{K}})\widetilde{V}_{\mathbf{\hat{m}}_{i}}^{*}P_{\mathcal{K}^{\perp}}\widetilde{V}^{(i)}(I_{E_{i}}\ot P_{\mathcal{K}}),
	\end{align*} where  $t_{i,\mathbf{\hat{m}}_{i}}: E_i \ot \mathbb{E}({\mathbf{\hat{m}}_{i}}) \to \mathbb{E}({\mathbf{\hat{m}}_{i}}) \ot E_i, \:\:\:\:i\in I_{k}$  are unitary maps which is coming from the unitary
	isomorphisms $t_{i,j}$.  This completes  the  proof of the lemma.\end{proof}
For each $i\in I_{k},$ we have  \begin{align}\label{positive} 0 &\leq (I_{E_{i}}\ot P_{\mathcal{K}})\widetilde{V}^{(i) *}P_{\mathcal{K}^{\perp}}\widetilde{V}^{(i)}(I_{E_{i}}\ot P_{\mathcal{K}}) \nonumber\\&= (I_{E_{i}}\ot P_{\mathcal{K}})\widetilde{V}^{(i) *}(I_{\mathcal{H}_{V}}-P_{\mathcal{K}})\widetilde{V}^{(i)}(I_{E_{i}}\ot P_{\mathcal{K}})\nonumber\\&=  (I_{E_{i}}\ot P_{\mathcal{K}})-
	\wT^{(i) *}\wT^{(i)} .
\end{align} Hence  $(I_{E_{i}}\ot P_{\mathcal{K}}-\wT^{(i) *}\wT^{(i)})^{\frac{1}{2}}$ exists and let it be denoted by  $\Delta{({T|_{\mathcal{K}}^{(i)}})}.$

The following lemma  relate $\Delta{({T|_{\mathcal{K}}^{(i)}})}$ to $[\wT^{(i)}, \wT_{\mathbf{\hat{m}}_{i}}^{*}].$  
\begin{lemma}\label{lemma2.5}
	Let $\mathbf{\hat{m}}_{i} \in \mathbb{N}^k_{0}\setminus\{\mathbf{0}\}.$ Then there exist contractions $ X_{\mathbf{\hat{m}}_{i}}:E_{i} \ot \mathcal{K} \to \mathbb{E}(\mathbf{\hat{m}}_{i})\ot \mathcal{K}$ and $Y_{\mathbf{\hat{m}}_{i}}:  \mathbb{E}(\mathbf{\hat{m}}_{i})\ot \mathcal{K} \to E_{i} \ot \mathcal{K} $  such that $[\wT^{(i)}, \wT_{\mathbf{\hat{m}}_{i}}^{*}]= X_{\mathbf{\hat{m}}_{i}}\Delta{({T|_{\mathcal{K}}^{(i)}})}$ and $[\wT_{\mathbf{\hat{m}}_{i}},\wT^{(i)*}]=\Delta{({T|_{\mathcal{K}}^{(i)}})} Y_{\mathbf{\hat{m}}_{i}}.$
\end{lemma}
\begin{proof}
	Using  Lemma \ref{2.2} we get \begin{align*}&
		\Delta{({T|_{\mathcal{K}}^{(i)}})}^{2}-[\wT^{(i)}, \wT_{\mathbf{\hat{m}}_{i}}^{*}]^{*}[\wT^{(i)}, \wT_{\mathbf{\hat{m}}_{i}}^{*}]\\&=(I_{E_{i}}\ot P_{\mathcal{K}})\widetilde{V}^{(i) *}P_{\mathcal{K}^{\perp}}\widetilde{V}^{(i)}(I_{E_{i}}\ot P_{\mathcal{K}})\\&\:\:\:\:\:\:\:-(I_{E_{i}}\ot P_{\mathcal{K}})\widetilde{V}^{(i)*}P_{\mathcal{K}^{\perp}}\widetilde{V}_{\mathbf{\hat{m}}_{i}}(I_{\mathbb{E}{(\mathbf{\hat{m}}_{i}})} \ot P_{\mathcal{K}})\widetilde{V}_{\mathbf{\hat{m}}_{i}}^{*}P_{\mathcal{K}^{\perp}}\widetilde{V}^{(i)}(I_{E_{i}}\ot P_{\mathcal{K}})\\&=(I_{E_{i}}\ot P_{\mathcal{K}})\widetilde{V}^{(i) *}P_{\mathcal{K}^{\perp}}[I_{\mathcal{H}_{V}}-\widetilde{V}_{\mathbf{\hat{m}}_{i}}(I_{\mathbb{E}{(\mathbf{\hat{m}}_{i})}} \ot P_{\mathcal{K}})\widetilde{V}_{\mathbf{\hat{m}}_{i}}^{*}] P_{\mathcal{K}^{\perp}}\widetilde{V}^{(i)}(I_{E_{i}}\ot P_{\mathcal{K}})\\& =((I_{E_{i}}\ot P_{\mathcal{K}})\widetilde{V}^{(i) *}P_{\mathcal{K}^{\perp}})[I_{\mathcal{H}_{V}}-\widetilde{V}_{\mathbf{\hat{m}}_{i}}(I_{\mathbb{E}{(\mathbf{\hat{m}}_{i})}} \ot P_{\mathcal{K}})\widetilde{V}_{\mathbf{\hat{m}}_{i}}^{*}] ((I_{E_{i}}\ot P_{\mathcal{K}})\widetilde{V}^{(i) *}P_{\mathcal{K}^{\perp}})^{*}.
	\end{align*} As $\widetilde{V}_{\mathbf{\hat{m}}_{i}}(I_{\mathbb{E}{(\mathbf{\hat{m}}_{i}})} \ot P_{\mathcal{K}})$ is a contraction, $
	I_{\mathcal{H}_{V}}-\widetilde{V}_{\mathbf{\hat{m}}_{i}}(I_{\mathbb{E}{(\mathbf{\hat{m}}_{i})}} \ot P_{\mathcal{K}})\widetilde{V}_{\mathbf{\hat{m}}_{i}}^{*} \geq 0,
	$ and hence \begin{equation*}
		\Delta{({T|_{\mathcal{K}}^{(i)}})}^{2}-[\wT^{(i)}, \wT_{\mathbf{\hat{m}}_{i}}^{*}]^{*}[\wT^{(i)}, \wT_{\mathbf{\hat{m}}_{i}}^{*}] \geq 0.
	\end{equation*}That is $[\wT^{(i)}, \wT_{\mathbf{\hat{m}}_{i}}^{*}]^{*}[\wT^{(i)}, \wT_{\mathbf{\hat{m}}_{i}}^{*}]\leq 	\Delta{({T|_{\mathcal{K}}^{(i)}})}^{2}.$ Therefore by Douglas's range inclusion theorem \cite{D96}, there exists a contraction $ X_{\mathbf{\hat{m}}_{i}}:E_{i} \ot \mathcal{K} \to \mathbb{E}(\mathbf{\hat{m}}_{i})\ot \mathcal{K}$ such that $ [\wT^{(i)}, \wT_{\mathbf{\hat{m}}_{i}}^{*}]=X_{\mathbf{\hat{m}}_{i}}\Delta({T|_{\mathcal{K}}^{(i)}}).$  Since  $ [\wT^{(i)},\wT_{\mathbf{\hat{m}}_{i}}^{*}]^{*}=[\wT_{\mathbf{\hat{m}}_{i}},\wT^{(i)*}],$ which proves the second equality.
\end{proof}
\begin{lemma}\label{lemma 2.4}
	Consider a  DCI-representation  $(\sigma, V^{(1)}, \dots, V^{(k)})$   of  $\mathbb{E}$ on  $\mathcal{H}_{V}$ and let $\mathcal{K}$ be a QS of $\mathcal{H}_{V}$. If \begin{equation}\label{hyp}
		(I_{E_{j}} \ot \Delta{({T|_{\mathcal{K}}^{(i)}})}(t_{i,j} \ot I_{\mathcal{H}_{V}})	(I_{E_{i}} \ot \Delta{({T|_{\mathcal{K}}^{(j)}})}=0,
	\end{equation} then, for each $\mathbf{\hat{m}}_{i},\mathbf{\hat{n}}_{j} \in \mathbb{N}^k_{0}\setminus\{\mathbf{0}\} ,$ \begin{enumerate}
		\item $(I_{E_j}\ot (I_{\mathbb{E}{(\mathbf{\hat{m}_i})}} \ot P_{\mathcal{K}}) \widetilde{V}_{\mathbf{\hat{m}}_{i}}^{*})X_{i,j}(I_{E_i} \ot \widetilde{V}_{\mathbf{\hat{n}}_{j}} (I_{\mathbb{E}{(\mathbf{\hat{n}}_j})} \ot P_{\mathcal{K}}))=0,
		$
		\item $ (I_{E_j}\ot (I_{E_{i}} \ot P_{\mathcal{K}})\widetilde{V}^{(i)*})X_{i,j}(I_{E_i} \ot \widetilde{V}_{\mathbf{\hat{n}}_{j}} (I_{\mathbb{E}{(\mathbf{\hat{n}}_j})} \ot P_{\mathcal{K}}))=0,$ and
		\item $(I_{E_j}\ot (I_{\mathbb{E}{(\mathbf{\hat{m}_i})}} \ot P_{\mathcal{K}}) \widetilde{V}_{\mathbf{\hat{m}}_{i}}^{*})X_{i,j} (I_{E_i} \ot \widetilde{V}^{(j)}( I_{E_{j}} \ot P_{\mathcal{K}}))=0.$
	\end{enumerate}
\end{lemma}
\begin{proof} (1)  By  Lemma \ref{lemma2.5}, there exist contractions $X_{\mathbf{\hat{m}}_{i}}$ and $Y_{\mathbf{\hat{n}}_{j}}$  such that $[\wT^{(i)}, \wT_{\mathbf{\hat{m}}_{i}}^{*}]= X_{\mathbf{\hat{m}}_{i}}\Delta{({T|_{\mathcal{K}}^{(i)}})}$ and $[\wT_{\mathbf{\hat{n}}_{j}},\wT^{(j)*}]= \Delta{({T|_{\mathcal{K}}^{(i)}})} Y_{\mathbf{\hat{n}}_{j}}$  and therefore  \begin{align*}& (I_{E_j} \ot[\wT^{(i)}, \wT_{\mathbf{\hat{m}}_{i}}^{*}] )(t_{i,j}\ot I_{\mathcal{H}_{V}})(I_{E_i} \ot[\wT^{(j)}, \wT_{\mathbf{\hat{n}}_{j}}^{*}]^{*} )\\&=
		(I_{E_j} \ot X_{\mathbf{\hat{m}}_{i}}\Delta{({T|_{\mathcal{K}}^{(i)}})} )(t_{i,j}\ot I_{\mathcal{H}_{V}})(I_{E_i} \ot \Delta{({T|_{\mathcal{K}}^{(i)}})} Y_{\mathbf{\hat{n}}_{j}})\\&=0,
	\end{align*} here the last equality follows by hypothesis (\ref{hyp}).
	
	On the other hand,  by using Lemma \ref{2.2}, we obtain
	\begin{align*}&
		(I_{E_j} \ot[\wT^{(i)}, \wT_{\mathbf{\hat{m}}_{i}}^{*}] )(t_{i,j}\ot I_{\mathcal{H}_{V}})(I_{E_i} \ot[\wT^{(j)}, \wT_{\mathbf{\hat{n}}_{j}}^{*}]^{*} )\\&=(I_{E_j} \ot (I_{\mathbb{E}{(\mathbf{\hat{m}}_{i}})} \ot P_{\mathcal{K}})\widetilde{V}_{\mathbf{\hat{m}}_{i}}^{*}P_{\mathcal{K}^{\perp}}\widetilde{V}^{(i)}(I_{E_{i}}\ot P_{\mathcal{K}}))(t_{i,j}\ot I_{\mathcal{H}_{V}})\\&\:\:\:\:\:\:\:(I_{E_i} \ot (I_{E_{j}}\ot P_{\mathcal{K}})\widetilde{V}^{(j)*}P_{\mathcal{K}^{\perp}}\widetilde{V}_{\mathbf{\hat{n}}_{j}}(I_{\mathbb{E}{(\mathbf{\hat{n}}_{j})}} \ot P_{\mathcal{K}}))\\&=(I_{E_j}\ot (I_{\mathbb{E}{(\mathbf{\hat{m}_i})}} \ot P_{\mathcal{K}}) \widetilde{V}_{\mathbf{\hat{m}}_{i}}^{*})X_{i,j}(I_{E_i} \ot \widetilde{V}_{\mathbf{\hat{n}}_{j}} (I_{\mathbb{E}{(\mathbf{\hat{n}}_{j})}} \ot P_{\mathcal{K}})).
	\end{align*} Thus, we get (1).
	
	(2) Using hypothesis ({\ref{hyp}}), we have \begin{align*}&
		0=(I_{E_{j}} \ot ((I_{E_{i}}\ot P_{\mathcal{K}})- \wT^{(i) *}\wT^{(i)}))(t_{i,j}\ot I_{\mathcal{H}_{V}})(I_{E_i} \ot[\wT^{(j)}, \wT_{\mathbf{\hat{n}}_{j}}^{*}]^{*} )\\&=(I_{E_{j}} \ot (I_{E_{i}}\ot P_{\mathcal{K}})\widetilde{V}^{(i) *}P_{\mathcal{K}^{\perp}}\widetilde{V}^{(i)}(I_{E_{i}}\ot P_{\mathcal{K}}))(t_{i,j}\ot I_{\mathcal{H}_{V}})\\&\:\:\:\:\:\:\:(I_{E_i} \ot (I_{E_{j}}\ot P_{\mathcal{K}})\widetilde{V}^{(j)*}P_{\mathcal{K}^{\perp}}\widetilde{V}_{\mathbf{\hat{n}}_{j}}(I_{\mathbb{E}{(\mathbf{\hat{n}}_{j})}} \ot P_{\mathcal{K}}))\\&= (I_{E_j}\ot( I_{E_{i}} \ot P_{\mathcal{K}}) \widetilde{V}^{(i)*})X_{i,j}(I_{E_i} \ot \widetilde{V}_{\mathbf{\hat{n}}_{j}}(I_{\mathbb{E}{(\mathbf{\hat{n}}_j})} \ot P_{\mathcal{K}})),
	\end{align*} where the previous equality follows from Lemma \ref{2.2}. Hence we proved (2).

	(3) Similarly, we can prove (3) as we proved (2). 
\end{proof}

For the proof of the main theorem, first, we show that $ \mathcal{K}$ reduces $\widetilde{V}^{(i)*}P_{\mathcal{K}^{\perp}}\widetilde{V}^{(i)},$ for each $ i \in I_{k}.$ For this it is enough to show that \begin{equation}\label{2.5}
	(I_{E_{i}}\ot P_{\mathcal{K}})(\widetilde{V}^{(i)*}P_{\mathcal{K}^{\perp}}\widetilde{V}^{(i)})=(\widetilde{V}^{(i)*}P_{\mathcal{K}^{\perp}}\widetilde{V}^{(i)})(I_{E_{i}}\ot P_{\mathcal{K}}), \:\:\:  i\in I_{k}.
\end{equation} 
Consider \begin{align*}
	(I_{E_{i}}\ot P_{\mathcal{K}})\widetilde{V}^{(i) *}P_{\mathcal{K}^{\perp}}\widetilde{V}^{(i)}(I_{E_{i}}\ot P_{\mathcal{K}})&=\widetilde{V}^{(i) *}P_{\mathcal{K}^{\perp}}\widetilde{V}^{(i)}(I_{E_{i}}\ot P_{\mathcal{K}})\\&\:\:\:\:\:\:\:-(I_{E_{i}}\ot  P_{\mathcal{K}^{\perp}})\widetilde{V}^{(i) *}P_{\mathcal{K}^{\perp}}\widetilde{V}^{(i)}(I_{E_{i}}\ot P_{\mathcal{K}})\\&=\widetilde{V}^{(i) *}P_{\mathcal{K}^{\perp}}\widetilde{V}^{(i)}(I_{E_{i}}\ot P_{\mathcal{K}})\\&\:\:\:\:\:\:\:-(I_{E_{i}}\ot  P_{\mathcal{K}^{\perp}})\widetilde{V}^{(i) *}\widetilde{V}^{(i)}(I_{E_{i}}\ot P_{\mathcal{K}})\\&=\widetilde{V}^{(i) *}P_{\mathcal{K}^{\perp}}\widetilde{V}^{(i)}(I_{E_{i}}\ot P_{\mathcal{K}}),
\end{align*} as $(I_{E_{i}}\ot  P_{\mathcal{K}^{\perp}})\widetilde{V}^{(i) *}P_{\mathcal{K}^{\perp}}= (I_{E_{i}}\ot  P_{\mathcal{K}^{\perp}})\widetilde{V}^{(i) *}.$  After taking adjoint both sides, we obtain \begin{equation*}\label{2.4}
	(I_{E_{i}}\ot P_{\mathcal{K}})\widetilde{V}^{(i) *}P_{\mathcal{K}^{\perp}}\widetilde{V}^{(i)}(I_{E_{i}}\ot P_{\mathcal{K}})=(I_{E_{i}}\ot P_{\mathcal{K}})\widetilde{V}^{(i) *}P_{\mathcal{K}^{\perp}}\widetilde{V}^{(i)},
\end{equation*}thus $ \mathcal{K}$ reduces $\widetilde{V}^{(i)*}P_{\mathcal{K}^{\perp}}\widetilde{V}^{(i)},$ for each $ i \in I_{k}.$

We are now prepared to begin the main section of the proof of Theorem \ref{Beurlin}.

\begin{proof}[Proof of Theorem \ref{Beurlin}]
	First assume  $\mathcal{K}$ to be a BQS of $\mathcal{H}_{V}.$  Therefore by using Theorem \ref{MT5} there exist a Hilbert space $\mathcal{H}_W,$ a pure DCI-representation  $(\pi, W^{(1)}, \dots, W^{(k)})$ of $\mathbb{E}$ on  $\mathcal{H}_W$ and an isometric multi-analytic operator $M_\Theta:{\mathcal{H}_W} \to \mathcal{H}_{V}$ such that
	$$\mathcal{K}=\mathcal{H}_{V}\ominus M_{\Theta}\mathcal{H}_W,$$
	where $\mathcal{W}_{\mathcal{H}_W}$ is the generating wandering subspace for $(\pi, W^{(1)}, \dots, W^{(k)}).$  
	Then $P_{\mathcal{K}^{\perp}}=M_{\Theta}M_{\Theta}^{*}.$ Using Equations (\ref{positive}) and (\ref{2.5}) we get the following equality \begin{align}\label{2.6}
		(I_{E_{i}}\ot P_{\mathcal{K}})- \wT^{(i) *}\wT^{(i)}&=(I_{E_{i}}\ot P_{\mathcal{K}})\widetilde{V}^{(i) *}P_{\mathcal{K}^{\perp}}\widetilde{V}^{(i)}(I_{E_{i}}\ot P_{\mathcal{K}})\nonumber\\&= (\widetilde{V}^{(i)*}P_{\mathcal{K}^{\perp}}\widetilde{V}^{(i)})(I_{E_{i}}\ot P_{\mathcal{K}}),
	\end{align} where $ i \in I_{k}.$ Note that  \begin{align*}&
		(I_{E_{j}}\ot ( (I_{E_{i}}\ot P_{\mathcal{K}})- \wT^{(i) *}\wT^{(i)}))(t_{i,j} \ot I_{\mathcal{H}_{V}})(I_{E_{i}}\ot ( (I_{E_{j}}\ot P_{\mathcal{K}})- \wT^{(j) *}\wT^{(j)}))\\&=(I_{E_{j}}\ot  \widetilde{V}^{(i)*}P_{\mathcal{K}^{\perp}}\widetilde{V}^{(i)}(I_{E_{i}}\ot P_{\mathcal{K}}))(t_{i,j} \ot I_{\mathcal{H}_{V}})(I_{E_{i}}\ot \widetilde{V}^{(j)*}P_{\mathcal{K}^{\perp}}\widetilde{V}^{(j)}(I_{E_{j}}\ot P_{\mathcal{K}}))\\&=(I_{E_{j}}\ot\widetilde{V}^{(i)*} P_{\mathcal{K}^{\perp}}\widetilde{V}^{(i)})(I_{E_{j}} \ot I_{E_{i}}\ot P_{\mathcal{K}})(t_{i,j} \ot I_{\mathcal{H}_{V}})\\&\:\:\:\:\:\:\:(I_{E_{i}}\ot \widetilde{V}^{(j)*}P_{\mathcal{K}^{\perp}}\widetilde{V}^{(j)}(I_{E_{j}}\ot P_{\mathcal{K}}))\\& =(I_{E_{j}}\ot\widetilde{V}^{(i)*} P_{\mathcal{K}^{\perp}}\widetilde{V}^{(i)})(t_{j,i} \ot I_{\mathcal{H}_{V}})(I_{E_{i}}\ot (I_{E_{j}}\ot P_{\mathcal{K}})  \widetilde{V}^{(j)*}P_{\mathcal{K}^{\perp}}\widetilde{V}^{(j)}(I_{E_{j}}\ot P_{\mathcal{K}}))\\&=(I_{E_{j}}\ot\widetilde{V}^{(i)*} P_{\mathcal{K}^{\perp}}\widetilde{V}^{(i)})(t_{j,i} \ot I_{\mathcal{H}_{V}})(I_{E_{i}}\ot \widetilde{V}^{(j)*}P_{\mathcal{K}^{\perp}}\widetilde{V}^{(j)}(I_{E_{j}}\ot P_{\mathcal{K}})),
	\end{align*} here the last inequality follows by (\ref{2.5}).
	Since ${\Theta}$ is an inner operator, $M_{\Theta}$ is an isometric multi-analytic, i.e., $M_{\Theta}^{*}M_{\Theta}=I_{\mathcal{H}_{W}} $ and \begin{align}\label{multi}
		M_{\Theta}W^{(i)}(\xi_i)h=V^{(i)}(\xi_i)M_{\Theta}h \hspace{0.7cm}\mbox{and} \hspace{0.7cm}M_{\Theta} \pi(a)h= \sigma(a)M_{\Theta}h,
	\end{align}
	where $\xi_i \in E_i, \: h \in \mathcal{H}_{W},$ $ a \in \mathcal{B}$ and $i \in I_{k}.$  Equation (\ref{multi}) provides that \begin{equation}\label{2.7}
		(I_{E_{j}} \ot M_{\Theta}^{*})\widetilde{V}^{(j)*}\widetilde{V}^{(i)}(I_{E_{i}}\ot M_{\Theta})=\widetilde{W}^{(j)*}\widetilde{W}^{(i)}.
	\end{equation}
	Now using  $P_{\mathcal{K}^{\perp}}=M_{\Theta}M_{\Theta}^{*}$, Equations (\ref{2.6}) and (\ref{2.7}), we obtain  \begin{align*}&
		(I_{E_{j}}\ot ( (I_{E_{i}}\ot P_{\mathcal{K}})- \wT^{(i) *}\wT^{(i)}))(t_{i,j} \ot I_{\mathcal{H}_{V}})(I_{E_{i}}\ot ( (I_{E_{j}}\ot P_{\mathcal{K}})- \wT^{(j) *}\wT^{(j)}))\\&=(I_{E_{j}}\ot\widetilde{V}^{(i)*} P_{\mathcal{K}^{\perp}}\widetilde{V}^{(i)})(t_{j,i} \ot I_{\mathcal{H}_{V}})(I_{E_{i}}\ot \widetilde{V}^{(j)*}P_{\mathcal{K}^{\perp}}\widetilde{V}^{(j)}(I_{E_{j}}\ot P_{\mathcal{K}}))\\&=(I_{E_{j}}\ot\widetilde{V}^{(i)*} M_{\Theta}M_{\Theta}^{*})\widetilde{V}^{(j)*}\widetilde{V}^{(i)}(I_{E_{i}}\ot M_{\Theta}M_{\Theta}^{*}\widetilde{V}^{(j)}( I_{E_{j}}\ot P_{\mathcal{K}}))\\&=(I_{E_{j}}\ot\widetilde{V}^{(i)*} M_{\Theta})\widetilde{W}^{(j)*}\widetilde{W}^{(i)}(I_{E_{i}}\ot M_{\Theta}^{*}\widetilde{V}^{(j)} (I_{E_{j}}\ot P_{\mathcal{K})})\\&=(I_{E_{j}}\ot\widetilde{V}^{(i)*} M_{\Theta})(I_{E_{j}} \ot \widetilde{W}^{(i)})(t_{j,i}\ot I_{\mathcal{H}_{W}})(I_{E_{i}}\ot\widetilde{W}^{(j)*})\\&\:\:\:\:\:\:\:(I_{E_{i}}\ot M_{\Theta}^{*}\widetilde{V}^{(j)} (I_{E_{j}}\ot P_{\mathcal{K}}))\\&=(I_{E_{j}}\ot\widetilde{V}^{(i)*})(I_{E_{j}}\ot M_{\Theta} \widetilde{W}^{(i)})(t_{j,i}\ot I_{\mathcal{H}_{W}})(I_{E_{i}}\ot\widetilde{W}^{(j)*} M_{\Theta}^{*})\\&\:\:\:\:\:\:\:(I_{E_{i}}\ot\widetilde{V}^{(j)}(I_{E_{j}}\ot P_{\mathcal{K}}))\\&=(I_{E_{j}}\ot I_{E_{i}}\ot M_{\Theta})(t_{j,i}\ot I_{\mathcal{H}_{W}})(I_{E_{i}}\ot(I_{E_{j}}\ot M_{\Theta}^{*})(I_{E_{j}}\ot P_{\mathcal{K}}))\\&=(t_{i,j}\ot I_{\mathcal{H}_{W}})(I_{E_{i}}\ot I_{E_{j}}\ot M_{\Theta} M_{\Theta}^{*})(I_{E_{i}}\ot I_{E_{j}}\ot P_{\mathcal{K}})=0.
	\end{align*}

	For the converse part, assume \begin{equation*}
		(I_{E_{j}}\ot ( (I_{E_{i}}\ot P_{\mathcal{K}})- \wT^{(i) *}\wT^{(i)}))(t_{i,j} \ot I_{\mathcal{ K}})(I_{E_{i}}\ot ( (I_{E_{j}}\ot P_{\mathcal{K}})- \wT^{(j) *}\wT^{(j)}))=0.
	\end{equation*} We have to show that  $\mathcal{K}$ is a  BQS of $\mathcal{H}_{V}$. For this, we need to show $X_{i,j}=0$ (using Lemma \ref{Lemma 2.1}). Define $\mathcal{L}:=\displaystyle\bigvee_{\mathbf{n} \in \mathbb{N}^k_0}\widetilde{V}_{\mathbf{n}}(\mathbb{E}(\mathbf{n})\ot \mathcal{K})$. First, we observe that $\mathcal{L}$ is a $(\sigma, V^{(1)}, \dots, V^{(k)})-$reducing subspace of $\mathcal{H}_{V}$. Indeed, one can easily see that $\mathcal{L}$ is a $(\sigma, V^{(1)}, \dots, V^{(k)})-$IS.  Now for reducing, only we have to show that $\widetilde{V}^{(i)*}\widetilde{V}_{\mathbf{n}}(\mathbb{E}(\mathbf{n})\ot \mathcal{K}) \subseteq  E_{i} \ot \mathcal{L}.$ Consider \begin{align*}
		\widetilde{V}^{(i)*}\widetilde{V}_{\mathbf{n}}(\mathbb{E}(\mathbf{n})\ot \mathcal{K})&=\widetilde{V}^{(i)*}\widetilde{V}^{(i)}_{n_{i}}(I_{E_{i}^{\ot n_i}}\ot \widetilde{V}_{\mathbf{\hat n}_{i}})(\mathbb{E}(\mathbf{n})\ot \mathcal{K})\\&=(I_{E_{i}} \ot \widetilde{V}^{(i)}_{n_{i}})(t_{i,n_{i}}\ot I_{\mathcal{H}_{V}})\\&\:\:\:\:\:\:\:(I_{E_{i}^{\ot n_i}}\ot \widetilde{V}^{(i)*})(I_{E_{i}^{\ot n_i}}\ot \widetilde{V}_{\mathbf{\hat{n}}_{i}})(\mathbb{E}(\mathbf{n})\ot \mathcal{K}) \\&\subseteq (I_{E_{i}} \ot \widetilde{V}^{(i)}_{n_{i}})(E_{i} \ot E_{i}^{\ot n_{i}}\ot \mathcal{K})\subseteq  E_{i} \ot \mathcal{L},
	\end{align*}  where $\mathbf{e}_{i}$ is multi-indices with $ 1$ in the $i^{th}$ place and zero elsewhere. Hence $\mathcal{L}$ reduces $(\sigma, V^{(1)}, \dots, V^{(k)}).$  Similarly $\mathcal{L}^{\perp}$ is also reducing subspace for\\ $(\sigma, V^{(1)}, \dots, V^{(k)}).$

	Therefore $(\sigma, V^{(1)}, \dots, V^{(k)})|_{\mathcal{L}}$ and $(\sigma, V^{(1)}, \dots, V^{(k)})|_{\mathcal{L}^{\perp}}$ are  pure DCI-representations of $\mathbb{E}$ on ${\mathcal{L}}$ and ${\mathcal{L}^{\perp}}, $ respectively. Then using \cite[Corollary 3.5]{HV21} there exist wandering subspaces $\mathcal{W}$ of $\mathcal{L}$ and  $\mathcal{W}_{0}$ of $\mathcal{L}^{\perp}$ such that \begin{equation*}
		\mathcal{L}=\bigoplus_{\mathbf{n} \in \mathbb{N}^k_0}\widetilde{V}_{\mathbf{n}}(\mathbb{E}(\mathbf{n})\ot \mathcal{W})\:\:\text{and}\:\: \mathcal{L}^{\perp}=\bigoplus_{\mathbf{n} \in \mathbb{N}^k_0}\widetilde{V}_{\mathbf{n}}(\mathbb{E}(\mathbf{n})\ot \mathcal{W}_{0}).
	\end{equation*} Note that \begin{align*}
		\bigoplus_{\mathbf{n} \in \mathbb{N}^k_0}\widetilde{V}_{\mathbf{n}}(\mathbb{E}(\mathbf{n})\ot \mathcal{W}_{\mathcal{H}_{V}})=\mathcal{H}_{V}=\mathcal{L}\oplus\mathcal{L}^{\perp}=\bigoplus_{\mathbf{n} \in \mathbb{N}^k_0}\widetilde{V}_{\mathbf{n}}(\mathbb{E}(\mathbf{n})\ot (\mathcal{W} \oplus \mathcal{W}_{0})).
	\end{align*} Thus, by the uniqueness of the generating wandering subspace, we get $\mathcal{W} \oplus \mathcal{W}_{0}=\mathcal{W}_{\mathcal{ H}_{V}}.$ Since $\mathcal{K}\subseteq\mathcal{L},$ $\mathcal{L}^{\perp}=\bigoplus_{\mathbf{n} \in \mathbb{N}^k_{0}}\widetilde{V}_{\mathbf{n}}(\mathbb{E}(\mathbf{n})\ot \mathcal{W}_{0})\subseteq \mathcal{K}^{\perp}.$ Hence, for each $ h \in \mathcal{L}^{\perp}$ \begin{equation*}
		(I_{E_{i}} \ot P_{\mathcal{K}})\widetilde{V}^{(i)*}P_{\mathcal{K}^{\perp}}h=(I_{E_{i}} \ot P_{\mathcal{K}})\widetilde{V}^{(i)*}h=0,
	\end{equation*} as $ \mathcal{L}^{\perp}$ reduces $(\sigma, V^{(1)}, \dots, V^{(k)}).$  This proves that \begin{equation*}
		X_{i,j}|_{E_{i} \ot {\mathcal{L}^{\perp}}}=0.
	\end{equation*} 
	
	Now we only need to show $X_{i,j}|_{E_{i} \ot {\mathcal{L}}}=0,$ which is equivalent to show $X_{i,j}(I_{E_{i}} \ot \widetilde{V}_{\mathbf{n}}(I_{\mathbb{E}({\mathbf{n}})}\ot P_{\mathcal{K}}))=0$
	for $ \mathbf{n}\in \mathbb{N}^k_{0}.$ Suppose $\mathbf{n}= \mathbf0$, then $X_{i,j}(I_{E_{i}} \ot \widetilde{V}_{\mathbf{n}}(I_{\mathbb{E}({\mathbf{n}})}\ot P_{\mathcal{K}}))=X_{i,j}(I_{E_{i}} \ot P_{\mathcal{K}})=0$ as $X_{i,j}=(I_{E_{j}} \ot P_{ \mathcal{K}^{\perp}})(I_{E_{j}}\ot \widetilde{V}^{(i)})(t_{i,j} \ot P_{\mathcal{K}})(I_{E_{i}}\ot \widetilde{V}^{(j)*})(I_{E_{i}}\ot P_{\mathcal{K}^\perp}).$ Hence we have to prove $X_{i,j}(I_{E_{i}} \ot \widetilde{V}_{\mathbf{n}}(I_{\mathbb{E}({\mathbf{n}})}\ot P_{\mathcal{K}}))=0 $ only for  $ \mathbf{n}\in \mathbb{N}^k_{0}\setminus\{\mathbf{0}\}.$ Furthermore for each $ h\in \mathcal{K}$, $ h_{0} \in \mathcal{L}^{\perp},$ $\xi_{i}\in E_{i} $ and $\eta_{\mathbf{n}}\in \mathbb{E}(\mathbf{n}), $ we have \begin{align*}&
		\langle X_{i,j}(I_{E_{i}} \ot \widetilde{V}_{\mathbf{n}})(\xi_{i}\ot \eta_{\mathbf{n}}\ot h), \xi_{j} \ot h_{0} \rangle\\&=\langle (I_{E_{j}} \ot P_{ \mathcal{K}^{\perp}} \widetilde{V}^{(i)})(t_{i,j} \ot P_{\mathcal{K}})(I_{E_{i}}\ot \widetilde{V}^{(j)*} P_{\mathcal{K}^\perp} \widetilde{V}_{\mathbf{n}})(\xi_{i}\ot \eta_{\mathbf{n}}\ot h),\xi_{j} \ot h_{0} \rangle\\&=\langle (t_{i,j} \ot P_{\mathcal{K}})(I_{E_{i}}\ot \widetilde{V}^{(j)*})(I_{E_{i}}\ot P_{\mathcal{K}^\perp})(I_{E_{i}} \ot \widetilde{V}_{\mathbf{n}})(\xi_{i}\ot \eta_{\mathbf{n}}\ot h)\\&\:\:\:\:\:\:\:,(I_{E_{j}}\ot \widetilde{V}^{(i)*}) \xi_{j} \ot h_{0}\rangle=0.
	\end{align*}  Thus $X_{i,j}(I_{E_{i}} \ot \widetilde{V}_{\mathbf{n}})(E_{i} \ot \mathbb{E}(\mathbf{n}) \ot \mathcal{K}) \perp E_{j} \ot \mathcal{L}^{\perp}.$ Therefore it suffices to show that \begin{equation*}
		X_{i,j}(I_{E_{i}} \ot \widetilde{V}_{\mathbf{n}})(E_{i} \ot \mathbb{E}(\mathbf{n}) \ot \mathcal{K}) \perp E_{j} \ot \mathcal{L},   
	\end{equation*} which is equivalent to prove that  \begin{equation*}
		X_{i,j}(I_{E_{i}} \ot \widetilde{V}_{\mathbf{n}})(E_{i} \ot \mathbb{E}(\mathbf{n}) \ot \mathcal{K}) \perp E_{j} \ot \widetilde{V}_{\mathbf{m}}(\mathbb{E}(\mathbf{m})\ot \mathcal{K}),\:\:\mathbf{m},\mathbf{n}\in \mathbb{N}^k_{0}\setminus\{\mathbf{0}\},
	\end{equation*}  as $\text{ran}  X_{i,j}\subseteq E_{j} \ot \mathcal{ K}^{\perp},$  the above equation is obvious  for $ \mathbf{m}= \mathbf0.$ For each $ h\in \mathcal{K}$, $\xi_{i}\in E_{i}$  and $\eta_{\mathbf{n}}\in \mathbb{E}(\mathbf{n}), $ \begin{align*}&
		\langle X_{i,j}(I_{E_{i}} \ot \widetilde{V}_{\mathbf{n}}(I_{\mathbb{E}({\mathbf{n}})}\ot P_{\mathcal{K}}))(\xi_{i}\ot \eta_{\mathbf{n}}\ot h),  I_{E_{j}} \ot \widetilde{V}_{\mathbf{m}}(I_{\mathbb{E}({\mathbf{m}})}\ot P_{\mathcal{K}})(\xi_{j}\ot \eta_{\mathbf{m}}\ot h)\rangle\\&=\langle (I_{E_{j}} \ot (I_{\mathbb{E}({\mathbf{m}})} \ot P_{\mathcal{K}}) \widetilde{V}_{\mathbf{m}}^{*}) X_{i,j}(I_{E_{i}} \ot \widetilde{V}_{\mathbf{n}}(I_{\mathbb{E}({\mathbf{n}})}\ot P_{\mathcal{K}}))(\xi_{i}\ot \eta_{\mathbf{n}}\ot h)\\&\:\:\:\:\:\:\:,  (\xi_{j}\ot \eta_{\mathbf{m}}\ot h)\rangle
	\end{align*} and hence we only need to show that \begin{equation}\label{main}
		(I_{E_{j}} \ot (I_{\mathbb{E}({\mathbf{m}})} \ot P_{\mathcal{K}}) \widetilde{V}_{\mathbf{m}}^{*}) X_{i,j}(I_{E_{i}} \ot \widetilde{V}_{\mathbf{n}}(I_{\mathbb{E}({\mathbf{n}})}\ot P_{\mathcal{K}}))=0.
	\end{equation} We start by showing the following equation \begin{equation}\label{main2}
		(I_{E_{j}}\ot\widetilde{V}^{(i)*})X_{i,j}(I_{E_{i}}\ot \widetilde{V}^{(j)})=0.
	\end{equation}By using hypothesis and Equation (\ref{positive}) we have 
	\begin{align*}&
		0= (I_{E_{j}}\ot ( (I_{E_{i}}\ot P_{\mathcal{K}})- \wT^{(i) *}\wT^{(i)}))(t_{i,j} \ot I_{\mathcal{H}_{V}})(I_{E_{i}}\ot ( (I_{E_{j}}\ot P_{\mathcal{K}})- \wT^{(j) *}\wT^{(j)}))\\&=(I_{E_{j}}\ot  (I_{E_{i}}\ot P_{\mathcal{K}})\widetilde{V}^{(i)*}P_{\mathcal{K}^{\perp}}\widetilde{V}^{(i)}(I_{E_{i}}\ot P_{\mathcal{K}}))(t_{i,j} \ot I_{\mathcal{H}_{V}})\\&\:\:\:\:\:\:\:(I_{E_{i}}\ot (I_{E_{j}}\ot P_{\mathcal{K}})\widetilde{V}^{(j)*}P_{\mathcal{K}^{\perp}}\widetilde{V}^{(j)}(I_{E_{j}}\ot P_{\mathcal{K}}))\\&=(I_{E_{j}}\ot  (I_{E_{i}}\ot P_{\mathcal{K}})\widetilde{V}^{(i)*}P_{\mathcal{K}^{\perp}}\widetilde{V}^{(i)})(t_{i,j} \ot I_{\mathcal{H}_{V}})\\&\:\:\:\:\:\:\:(I_{E_{i}}\ot (I_{E_{j}}\ot P_{\mathcal{K}})\widetilde{V}^{(j)*}P_{\mathcal{K}^{\perp}}\widetilde{V}^{(j)}(I_{E_{j}}\ot P_{\mathcal{K}}))\\&=(I_{E_{j}}\ot  \widetilde{V}^{(i)*}P_{\mathcal{K}^{\perp}}\widetilde{V}^{(i)}(I_{E_{i}}\ot P_{\mathcal{K}}))(t_{i,j} \ot I_{\mathcal{H}_{V}})(I_{E_{i}}\ot (I_{E_{j}}\ot P_{\mathcal{K}})\widetilde{V}^{(j)*}P_{\mathcal{K}^{\perp}}\widetilde{V}^{(j)})\\&=(I_{E_{j}}\ot  \widetilde{V}^{(i)*}P_{\mathcal{K}^{\perp}}\widetilde{V}^{(i)})(t_{i,j} \ot P_{\mathcal{K}})(I_{E_{i}}\ot \widetilde{V}^{(j)*}P_{\mathcal{K}^{\perp}}\widetilde{V}^{(j)})\\&= (I_{E_{j}}\ot\widetilde{V}^{(i)*})X_{i,j}(I_{E_{i}}\ot \widetilde{V}^{(j)}).
	\end{align*} Now for  $ \mathbf{m}, \mathbf{n}\in \mathbb{N}^k_{0}\setminus\{\mathbf{0}\}$ with $m_{i},n_{j}\neq0$ and using $\widetilde{V}_{\mathbf{m}}=\widetilde{V}^{(i)}(I_{E_{i}} \ot \widetilde{V}_{{m_{i}\mathbf{e}_{i}-\mathbf{e}_{i}}})(I_{\mathbb{E}({m_{i}\mathbf{e}_i})}\ot \widetilde{V}_{\mathbf{\hat{m}}_{i}})$ and  $\widetilde{V}_{\mathbf{n}}=\widetilde{V}^{(j)}(I_{E_{j}} \ot \widetilde{V}_{{n_{j}\mathbf{e}_{j}-\mathbf{e}_{j}}})(I_{\mathbb{E}({n_{j}\mathbf{e}_{j}})}\ot \widetilde{V}_{\mathbf{\hat{n}}_{j}})$ in the left hand side of the Equation (\ref{main}), we obtain \begin{align*}&
		(I_{E_{j}} \ot ( I_{\mathbb{E}({\mathbf{m}})} \ot P_{\mathcal{K}}) \widetilde{V}_{\mathbf{m}}^{*}) X_{i,j}(I_{E_{i}} \ot \widetilde{V}_{\mathbf{n}}(I_{\mathbb{E}({\mathbf{n}})}\ot P_{\mathcal{K}}))\\&=(I_{E_{j}} \ot (I_{\mathbb{E}({\mathbf{m}})} \ot P_{\mathcal{K}})(I_{E({m_{i}\mathbf{e}_i})}\ot \widetilde{V}_{\mathbf{\hat{m}}_{i}} ^{*})(I_{E_{i}} \ot \widetilde{V}_{{m_{i}\mathbf{e}_{i}-\mathbf{e}_{i}}}^{*})\widetilde{V}^{(i)*})X_{i,j}\\&\:\:\:\:\:\:\:\:\:\:\:\:\:\:(I_{E_{i}} \ot\widetilde{V}^{(j)}(I_{E_{j}} \ot \widetilde{V}_{{n_{j}\mathbf{e}_{j}-\mathbf{e}_{j}}})(I_{\mathbb{E}({n_{j}\mathbf{e}_{j}})}\ot \widetilde{V}_{\mathbf{\hat{n}}_{j}})(I_{\mathbb{E}({\mathbf{n}})}\ot P_{\mathcal{K}}))=0,
	\end{align*} here the last equality follows by the Equation (\ref{main2}). Now we will discuss the remaining cases : $ \mathbf{m}, \mathbf{n}\in \mathbb{N}^k_{0}\setminus\{\mathbf{0}\}$ where 
	\begin{enumerate}
		\item $m_{i}=n_{j}=0$, 
		\item $m_{i}\neq 0$ and $n_{j}=0,$ 
		\item $m_{i}= 0$, $n_{j}\neq 0$. 
	\end{enumerate} 
	In the first case, when we substitute these values on the left-hand side of the Equation (\ref{main}), we get \begin{align*}&(I_{E_{j}} \ot I_{\mathbb{E}({\mathbf{m}})} \ot P_{\mathcal{K}})(I_{E_{j}} \ot \widetilde{V}_{\mathbf{m}}^{*}) X_{i,j}(I_{E_{i}} \ot \widetilde{V}_{\mathbf{n}}(I_{\mathbb{E}({\mathbf{n}})}\ot P_{\mathcal{K}}))\\&=
		(I_{E_{j}} \ot I_{\mathbb{E}({\mathbf{\hat{m}}_{i}})} \ot P_{\mathcal{K}})(I_{E_{j}} \ot \widetilde{V}_{\mathbf{\hat{m}}_{i}}^{*}) X_{i,j}(I_{E_{i}} \ot \widetilde{V}_{\mathbf{\hat{n}}_{j}}(I_{\mathbb{E}({\mathbf{\hat{n}}_{j}})}\ot P_{\mathcal{K}}))=0,
	\end{align*} by using first part of Lemma \ref{lemma 2.4}.
	
	For the (2), substitute $\widetilde{V}_{\mathbf{m}}=\widetilde{V}^{(i)}(I_{E_{i}} \ot \widetilde{V}_{\mathbf{m}-\mathbf{e}_{i}})$,  $m_{i}\neq 0$ and $n_{j}=0$ in left hand side of Equation (\ref{main}), we get \begin{align*}&
		(I_{E_{j}} \ot I_{\mathbb{E}({\mathbf{m}})} \ot P_{\mathcal{K}} (I_{E_{i}} \ot \widetilde{V}_{\mathbf{m}-\mathbf{e}_{i}}^{*})\widetilde{V}^{(i)*}) X_{i,j}(I_{E_{i}} \ot \widetilde{V}_{\mathbf{\hat{n}}_{i}}(I_{\mathbb{E}({\mathbf{\hat{n}}_{i}})}\ot P_{\mathcal{K}}))\\&=(I_{E_{j}} \ot I_{\mathbb{E}({\mathbf{m}})} \ot P_{\mathcal{K}}(I_{E_{i}} \ot \widetilde{V}_{\mathbf{m}-\mathbf{e}_{i}}^{*}) (I_{E_{i}}\ot P_{\mathcal{K}})\widetilde{V}^{(i)*}) X_{i,j}(I_{E_{i}} \ot \widetilde{V}_{\mathbf{\hat{n}}_{i}}(I_{\mathbb{E}({\mathbf{\hat{n}}_{i}})}\ot P_{\mathcal{K}}))\\&=0,
	\end{align*} by using second part of Lemma \ref{lemma 2.4}. 
	
	We can easily demonstrate the Equation (\ref{main}) for $m_{i}= 0$ and $n_{i}\neq 0$ by using the same method as in (2) and the third part of Lemma \ref{lemma 2.4}. Hence $\mathcal{K}$ is a  BQS of $\mathcal{H}_{V}.$ This completes the proof. 
\end{proof}

\section{Model theory and concrete regular dilation for Brehmer-Solel tuples}\label{3}

Consider the following natural question: Which type of covariant representation $(\sigma, V)$ of $E$ on $\mathcal{H}_{V}$ is unitarily equivalent to the restriction of the induced representation  $(\rho, S)$ of $E$ to some QS of $\mathcal{F}(E)\otimes_{\pi}\mathcal{H}$? The following Sz. Nagy-Foias type dilation which is based on \cite[Theorem 3]{AC14} and \cite[Proposition 10, Lemma 11]{MS98,MS09} gives answer to the above question:
\begin{theorem}[Muhly-Solel]\label{dilation}
	Let $(\sigma, V)$  be a pure, completely contractive covariant representation of $E$ on $\mathcal{H}_{V}.$ Then $(\sigma, V)$ is unitarily equivalent to  the  restriction of the induced representation $(\rho, S)$ to $\Pi_V \mathcal{H}_{V},$ a QS of $\mathcal{F}(E)\otimes_{\pi}\mathcal{H}$, where $\Pi_V:\mathcal{H}_{V}\to \mathcal{F}(E)\otimes_{\pi}\mathcal{H}$, is an isometry.
\end{theorem} 
Observe that from Theorem \ref{hs} it  follows that $\mathcal{K}=\Pi_V \mathcal{H}_{V}$ is a BQS of $\mathcal{F}(E)\ot \mathcal{H},$ therefore there exist a Hilbert space $\mathcal{H}_{T},$ a pure isometric representation $(\mu, T)$ of $E$ on $\mathcal{H}_{T}$ and an isometric multi-analytic operator $M_\Theta: { \mathcal{H}_{T}} \to \mathcal{F}(E)\ot \mathcal{H}$ such that $$\mathcal{K}=\mathcal{F}(E)\ot \mathcal{H}\ominus M_{\Theta} \mathcal{H}_{T}.$$ Hence $\Pi_{V}$ is  an isometry if and only if  $\lim_{n \rightarrow \infty}\|\widetilde{V}_{n}h\|^{2}$, i.e., $(\sigma,V)$ is pure. Since  $\mathcal{K}:=\Pi_V \mathcal{H}_{V}$ is IS under $\widetilde{S}^{*}$, $\mathcal{K}$ is a QS of $\mathcal{F}(E)\ot \mathcal{H}.$  Therefore $(\sigma, V)$ is isomorphic to  $(\rho', S')$ where $\rho'(a)=P_{\mathcal{ K}}\rho(a)|_{{\mathcal{K}}}$ and $S'(\xi)=P_{\mathcal{K}}{S}(\xi)|_{\mathcal{K}},$ for $a\in \mathcal{B}$ and $ \xi \in E,$ equivalently \begin{equation}\label{imp1}
	(\sigma, V)\cong(\rho', S').
\end{equation} In other words, we can say a pure, completely contractive covariant representation $(\sigma, V)$ of $E$ on $\mathcal{H}_{V}$ dilates to an induced representation $(\rho, S).$

In this section, based on \cite{AC14, S008,S009} we generalize the following one of the most concrete multivariable dilation results from \cite{C3, CV95} :

\begin{thm}\label{canonical dilation}
	If $T=(T_1,\ldots,T_n)$ is a Brehmer tuple on $\mathcal{ H}$, then $T$ dilates to $(M_{z_1}, \ldots, M_{z_n})$ on $H^2_{\mathcal{D}}(\mathbb{D}^n)$.
\end{thm}

We start with the definition of induced representation, which is a generalization of the multiplication operators $M_{z_i} \otimes I_{\mathcal{H}}$ on the vector-valued Hardy space $H^2_{\mathcal{H}}(\mathbb{D}^k).$  
\begin{definition}
	Let $\pi$ be a representation of $\mathcal{B}$ on the Hilbert space $\mathcal{H}.$  Define an isometric representation $(\rho, S^{(1)},S^{(2)}, \dots , S^{(k)})$ of $\mathbb{E}$ on $\mathcal{F}(\mathbb{E})\otimes_{\pi} \mathcal{H}$ (cf. \cite{SZ08}) by  $$\rho(a)=\phi_{\infty}(a) \otimes I_{\mathcal{H}}, \:\:\mbox{and}\:\: S^{(i)}(\xi_i)=V_{{\xi}_i} \otimes I_{\mathcal{H}},\: i\in I_{k}, \xi_i \in E_i, a \in \mathcal{B},$$ where $V_{\xi_i}$  denotes the {\rm creation operator} on $\mathcal{F}(\mathbb{E})$ determined by $\xi_i,$ that is, $V_{\xi_i}(\eta)=\xi_i \otimes \eta,$ where $\eta \in \mathcal{F}(\mathbb{E})$ and  $\phi_{\infty}$  denotes the canonical left action of $\mathcal{B}$ on $\mathcal{F}(\mathbb{E}).$ It is easy to see that  $(\rho, S^{(1)}, \dots, S^{(k)})$ is doubly commuting, and it is called {\rm induced representation} of $\mathbb{E}$ induced by $\pi.$ Any covariant representation of $\mathbb{E}$ which is isomorphic to $(\rho,S^{(1)}, \dots ,S^{(k)})$ is also called an {\rm induced representation}.
\end{definition}
\begin{definition}
	Let $(\sigma, V^{(1)}, V^{(2)},\dots, V^{(k)})$  be the  completely contractive covariant representation of $\mathbb{E}$ on   $\mathcal{H}_{V}$ and  $(\mu, T^{(1)}, T^{(2)},\dots T^{(k)})$  isometric representation of $\mathbb{E}$ on  the Hilbert space $\mathcal{H}_{T}$ . We say that   $(\mu, T^{(1)}, T^{(2)},\dots T^{(k)})$ is an {\rm isometric dilation} (cf. \cite{S009}) of $(\sigma, V^{(1)}, V^{(2)},\dots, V^{(k)})$  if there exists an isometry $\Pi: \mathcal{H}_{V} \rightarrow \mathcal{H}_{T}$ such that $\Pi \sigma(a)=\mu(a)\Pi, $ and $(I_E \otimes \Pi)\wt{V}^{(i)*}=\wt{T}^{(i)*}\Pi, $ for $a \in \mathcal{B}$ and $i \in I_{k}.$ 
\end{definition}

We are looking for a covariant representation $(\sigma, V^{(1)}, V^{(2)},\dots, V^{(k)})$ of $\mathbb{E}$ on  $\mathcal{H}_{V} $ that is unitarily equivalent to the restriction of the induced representation $(\rho, S^{(1)}, \dots, S^{(k)})$ to QS of $\mathcal{F}(\mathbb{E})\otimes\mathcal{H}.$ 

Since   $(\rho,S^{(1)}, \dots ,S^{(k)})$ is completely contractive, the representation $(\sigma, V^{(1)}, V^{(2)},\dots, V^{(k)})$ of $\mathbb{E}$ on $\mathcal{H}_{V}$ is  completely contractive.   
Therefore, consider a completely contractive covariant representation $(\sigma, V^{(1)}, V^{(2)},\dots, V^{(k)})$ of $\mathbb{E}$ on  $\mathcal{H}_{V}.$ We need to prove that there exists a isometry $\Pi_V:\mathcal{H}_{V}\to \mathcal{F}(\mathbb{E})\otimes_{\pi}\mathcal{H}$ such that 
\begin{align}\label{dilate}
	&\Pi_V\sigma(a)=\rho(a)\Pi_V \:\:\:\:\mbox{and}\:\:\:\: (I_{E_{i}} \ot \Pi_V)\widetilde{V}^{(i)*}=\widetilde{S}^{(i)*}\Pi_V,  
\end{align} where  $a \in \mathcal{B}$ and $ i \in I_{k}.$ If such an isometry exists, then the needed QS of  $\mathcal{F}(\mathbb{E})\ot \mathcal{H}$ will equal its range, $\Pi_{V}(\mathcal{ H}_{V}).$
Suppose such an isometry $\Pi_{V}$ exists. Note that $\Pi_{V}=\bigoplus_{\mathbf{n}\in\mathbb{N}_{0}^{k}}\Pi_V^{\mathbf{n}},$ where $\Pi_V^\mathbf{n}:\mathcal{H}_{V}\to \mathbb{E}({\mathbf{n}})\ot \mathcal{H}$ defined by $\Pi_V^\mathbf{n}h=P_{{\mathbb{E}(\mathbf{n})}\ot \mathcal{H}}\Pi_Vh, h \in \mathcal{H}_{V}.$ Now using Equation(\ref{dilate}), we obtain \begin{equation}\label{regular}
	\Pi_V^{\mathbf{n}+\mathbf{e}_{i}}=(I_{E_{i} } \ot \Pi_V^{\mathbf{n}})\widetilde{V}^{(i)* }
\end{equation} for every $\mathbf{n} \in \mathbb{N}_{0}^{k}$ and $ i \in I_{k}.$  Indeed \begin{align*}
	(I_{E_{i} } \ot \Pi_V^{\mathbf{n}})\widetilde{V}^{(i)* }&=(I_{E _{i}} \ot P_{{\mathbb{E}(\mathbf{n})}\ot \mathcal{H}})(I_{E _{i}} \ot \Pi_V)\widetilde{V}^{(i)*} =(I_{E _{i}} \ot P_{{\mathbb{E}(\mathbf{n})}\ot \mathcal{H}})\widetilde{S}^{(i)*}\Pi_V\\&=(I_{E_{i}}\ot  P_{{\mathbb{E}(\mathbf{n})}\ot \mathcal{H}})\Pi_V= P_{{\mathbb{E}(\mathbf{n}+\mathbf{e}_{i})}\ot \mathcal{H}}\Pi_V=\Pi_V^{\mathbf{n}+\mathbf{e}_{i}},
\end{align*} where $\mathbf{e}_{i}:=(0,0,\dots,1,0,\dots 0),$ i.e., $1$ at $i^{th}$ place and $0$ otherwise. More generally, for $ {m_{i}}\in \mathbb{N}_{0} $ we have \begin{align*} 
	\Pi_V^{\mathbf{n}+{m_{i}}\mathbf{e}_{i}}=(I_{E_{i}^{\ot m_i} } \ot \Pi_V^{\mathbf{n}})\widetilde{V}^{(i)* }_{m_i},
\end{align*}
and hence
\begin{align*}
	\Pi_V^{\mathbf{n}}=(I_{\mathbb{E}(\mathbf{n})}\ot \Pi_V^{\mathbf{0}})\widetilde{V}_{\mathbf{n}}^{*}.
\end{align*}
Indeed, for $\mathbf{n}=({n_1,n_2,\dots,n_{k}}) \in \mathbb{N}_{0}^{k}$ we have
\begin{align*}
	\Pi_V^{\mathbf{n}}&=(I_{E_{1}^{\ot {{n_{1}}}} } \ot \Pi_V^{\mathbf{n}-{n_{1}}\mathbf{e}_1})\widetilde{V}^{(1)* }_{{n_{1}}}\\&=(I_{E_{1}^{\ot{{n_1}}} } \ot (I_{E_{2}^{\ot {{n_{2}}}}}\ot\Pi_V^{\mathbf{n}-({n_1}\mathbf{e}_1+{n_{2}}\mathbf{e}_2)} )\widetilde{V}^{(2)* }_{{n_{2}}})\widetilde{V}^{(1)* }_{{n_{1}}}\\&=(I_{E_{1}^{\ot {{n_{1}}}} } \ot (I_{E_{2}^{\ot {{n_{2}}}}}\ot(I_{E_{3}^{\ot {{n_{3}}}}}\ot \Pi_V^{\mathbf{n}-({n_{1}}\mathbf{e}_1+{n_{2}}\mathbf{e}_2+{n_{3}}\mathbf{e}_3)}) \widetilde{V}^{(3)* }_{{n_{3}}} )\widetilde{V}^{(2)* }_{{n_{2}}})\widetilde{V}^{(1)* }_{{n_{1}}}=\cdots\\&=(I_{\mathbb{E}(\mathbf{n})}\ot \Pi_V^{\mathbf{0}})\widetilde{V}_{\mathbf{n}}^{*}.
\end{align*}Thus we get $ \Pi_V^{\mathbf{n}}=(I_{\mathbb{E}(\mathbf{n})}\ot \Pi_V^{\mathbf{0}})\widetilde{V}_{\mathbf{n}}^{*}.$ Therefore,  it is  sufficient to find  $ \Pi_V^{\mathbf{0}}$ only, since the remaining $\Pi_V^{\mathbf{n}}$ will be determined by $ (I_{\mathbb{E}(\mathbf{n})}\ot \Pi_V^{\mathbf{0}})\widetilde{V}_{\mathbf{n}}^{*}.$ Since $\Pi_V$ is an isometry, we have \begin{align}
	\|h\|^{2}=\| \Pi_Vh\|^{2}&=\sum_{\mathbf{n} \in \mathbb{N}_{0}^{k}}\| \Pi_V^{\mathbf{n}}h\|^{2}=\sum_{\mathbf{n} \in \mathbb{N}_{0}^{k}}\|(I_{\mathbb{E}(\mathbf{n})}\ot \Pi_V^{\mathbf{0}})\widetilde{V}_{\mathbf{n}}^{*}h\|^{2},\:\:\:\:\:h \in \mathcal{H}_{V}.
\end{align}
Suppose $u=\{u_1,u_2,\dots, u_r\}\subseteq I_{k}$ such that $u_1< u_2< \dots < u_r.$ Define $\mathbf{e}(u):=\mathbf{e}_{u_1}+\mathbf{e}_{u_2}+\dots+\mathbf{e}_{u_r}.$ 
Note that  \begin{align*}
	(I_{\mathbb{E}(\mathbf{e}(u))} \ot \Pi_{V})\widetilde{V}^{*}_{\mathbf{e}(u)}= (I_{\mathbb{E}(\mathbf{e}(u))}  \ot\bigoplus_{\mathbf{n} \in \mathbb{N}_{0}^{k}}(I_{\mathbb{E}(\mathbf{n})}\ot \Pi_V^{\mathbf{0}})\widetilde{V}_{\mathbf{n}}^{*})\widetilde{V}^{*}_{\mathbf{e}(u)}.
\end{align*}
Since $(I_{\mathbb{E}(\mathbf{e}(u))} \ot \Pi_{V})$ is an isometry, we get \begin{align}&\label{F}
	\|\widetilde{V}_{\mathbf{e}(u)}^{*}h\|^{2}=\|(I_{\mathbb{E}(\mathbf{e}(u))} \ot \Pi_{V})\widetilde{V}_{\mathbf{e}(u)}^{*}h\|^{2}\nonumber\\&=\sum_{\mathbf{n} \in \mathbb{N}_{0}^{k}}\|(I_{\mathbb{E}(\mathbf{e}(u))} \ot(I_{\mathbb{E}(\mathbf{n})}\ot \Pi_V^{\mathbf{0}})\widetilde{V}_{\mathbf{n}}^{*})\widetilde{V}^{*}_{\mathbf{e}(u)}h\|^{2},\:\:\:h \in \mathcal{H}_{V}.
\end{align} Denote by $|u|$ the cardinality of $u.$ It follows from Equation ({\ref{F}}) that \begin{align*}&
	\sum_{u\subseteq I_{k}}(-1)^{|u|}\|\widetilde{V}_{\mathbf{e}(u)}^{*}h\|^{2}=\sum_{u\subseteq I_{k}}(-1)^{|u|}\left(\sum_{\mathbf{n} \in \mathbb{N}_{0}^{k}}\| (I_{\mathbb{E}(\mathbf{e}(u))}   \ot(I_{\mathbb{E}(\mathbf{n})}\ot \Pi_V^{\mathbf{0}})\widetilde{V}_{\mathbf{n}}^{*})\widetilde{V}^{*}_{\mathbf{e}(u)}h\|^{2}\right)\\&=\sum_{\mathbf{m} \in \mathbb{N}^k_0}\left({\sum_{{\mathbf{n}\leq\mathbf{m}},{\max\{m_j-n_j\}\leq 1}}}(-1)^{|\mathbf{m}|-|\mathbf{n}|} \| (I_{\mathbb{E}(\mathbf{m})}\ot \Pi_V^{\mathbf{0}})\widetilde{V}_{\mathbf{m}}^{*}h\|^{2}\right)\\&=\|\Pi_V^{\mathbf{0}}h\|^{2}.
\end{align*}
Hence \begin{align*}
	\sum_{u\subseteq I_{k}}(-1)^{|u|}\|\widetilde{V}_{\mathbf{e}(u)}^{*}h\|^{2}\geq 0,\:\:\:\:h\in \mathcal{H}_{V},
\end{align*}  equivalently \begin{align*}
	\sum_{u\subseteq I_{k}}(-1)^{|u|}\widetilde{V}_{\mathbf{e}(u)}\widetilde{V}_{\mathbf{e}(u)}^{*}\geq 0.
\end{align*}
Now, define  $\Delta_{*}(V)=(\sum_{u\subseteq I_{k}}(-1)^{|u|}\widetilde{V}_{\mathbf{e}(u)}\widetilde{V}_{\mathbf{e}(u)}^{*})^{\frac{1}{2}}$ and $\mathcal{D}_{*,V}=\overline{Im\Delta_*(V)}.$ Therefore we can choose   for $\mathcal{H}$ is $\mathcal{D}_{*,V}$ and  $\Pi_V^{\mathbf{0}}=\Delta_{*}(V),$ thus  $ \Pi_V^{\mathbf{n}}=(I_{\mathbb{E}(\mathbf{n})}\ot \Delta_{*}(V))\widetilde{V}_{\mathbf{n}}^{*}.$ Then for each $ h \in \mathcal{H}_{V},$ we obtain \begin{align*}
	\| \Pi_Vh\|^{2}&=\sum_{\mathbf{n} \in \mathbb{N}_{0}^{k}}\| \Pi_V^{\mathbf{n}}h\|^{2}=\lim_{p\rightarrow \infty}\sum_{{\mathbf{n} \in \mathbb{N}_{0}^{k}},{\max n_j\leq p-1}}\|(I_{\mathbb{E}(\mathbf{n})}\ot \Pi_V^{\mathbf{0}})\widetilde{V}_{\mathbf{n}}^{*}h\|^{2}\\&=\lim_{p \rightarrow \infty}\sum_{{\mathbf{n} \in \mathbb{N}_{0}^{k}},{\max n_j\leq p-1}}\sum_{{u}\subseteq{I_{k}}}(-1)^{|u|}\|(I_{\mathbb{E}(\mathbf{n})}\ot\widetilde{V}^{*}_{\mathbf{e}(u)} )\widetilde{V}_{\mathbf{n}}^{*}h\|^{2}\\&=\lim_{p \rightarrow \infty}\sum_{{\mathbf{m} \in \mathbb{N}_{0}^{k}},{\max m_j\leq p}}\left( \sum_{\mathbf{n}\leq \mathbf{m},\max \{{m_j-n_j}\}\leq 1,{\max n_j\leq p-1}}(-1)^{|\mathbf{m}|-|\mathbf{n}|}\|\widetilde{V}_{\mathbf{m}}^{*}h\|^{2}\right).
\end{align*} Observe that \begin{align*}
	\sum_{\mathbf{n}\leq \mathbf{m},\max \{{m_j-n_j}\}\leq 1,{\max n_j\leq p-1}}(-1)^{|\mathbf{m}|-|\mathbf{n}|}=0,
\end{align*} whenever $0< m_i< p,$ for all $i \in I_{k}.$ If $m_i=0$ or $m_i=p,$ then \begin{align*}
	\sum_{\mathbf{n}\leq \mathbf{m},\max \{{m_j-n_j}\}\leq 1,{\max n_j\leq p-1}}(-1)^{|\mathbf{m}|-|\mathbf{n}|}=(-1)^{|{\supp\mathbf{m}}|}.
\end{align*} Therefore
\begin{align*}
	\| \Pi_Vh\|^{2}&=\lim_{p \rightarrow \infty}\sum_{u\subseteq I_{k}}(-1)^{|u|}\|\widetilde{V}_{\mathbf{e}(u)}^{*p}h\|^{2}\\&=\|h\|^{2}+\lim_{p \rightarrow \infty}\sum_{\emptyset\neq {u\subseteq I_{k}}}(-1)^{|u|}\|\widetilde{V}_{\mathbf{e}(u)}^{*p}h\|^{2},\:\:\:\: h \in \mathcal{ H}_{V}.
\end{align*}
Hence $\Pi_{V}$ will be an isometry if $SOT-\lim_{p \rightarrow \infty}\widetilde{V}^{({j}) *} _p=0$ for $j\in I_{k}.$

Now we define the following Brehmer-Solel condition.
\begin{definition}
	Let $(\sigma, V^{(1)}, V^{(2)},\dots, V^{(k)})$ be  a pure, completely contractive representation of $\mathbb{E}$ on $\mathcal{H}_{V}.$ Then  we say it is satisfy {\rm Brehmer-Solel condition} (cf. \cite{S008}) if it satisfies the following condition\begin{align*}
		\sum_{u \subseteq I_k}(-1)^{|u|}\widetilde{V}_{\mathbf{e}(u)}\widetilde{V}_{\mathbf{e}(u)}^{*}\geq 0.
	\end{align*}
\end{definition} 
The above observations yield the following theorem: a generalization of \cite[Theorem 4]{AC14}.
\begin{theorem}\label{6}
	Let $(\sigma, V^{(1)}, V^{(2)},\dots, V^{(k)})$  be a  completely contractive covariant representation of $\mathbb{E}$ on $\mathcal{H}_{V},$ satisfying Brehmer-Solel condition. Then there exists  $\Pi_V:\mathcal{H}_{V}\to \mathcal{F}(\mathbb{E})\otimes_{\pi}\mathcal{D}_{ *,V},$ satisfy \begin{align*}
		\Pi_V\sigma(a)=\rho(a)\Pi_V \:\:\:\:\mbox{and}\:\:\:\: (I_{E_{i}} \ot \Pi_V)\widetilde{V}^{(i)*}=\widetilde{S}^{(i)*}\Pi_V,  
	\end{align*} where  $a \in \mathcal{B}$ and $ i \in I_{k}$  such that \begin{align*}
		\| \Pi_Vh\|^{2}&=\lim_{p \rightarrow \infty}\sum_{u\subseteq I_{k}}(-1)^{|u|}\|\widetilde{V}_{\mathbf{e}(u)}^{*p}h\|^{2}\\&=\|h\|^{2}+\lim_{p \rightarrow \infty}\sum_{\emptyset\neq {u\subseteq I_{k}}}(-1)^{|u|}\|\widetilde{V}_{\mathbf{e}(u)}^{*p}h\|^{2},\:\:\:\: h \in \mathcal{ H}_{V}.
	\end{align*}
\end{theorem}
In particular, if we take a pure, completely contractive covariant representation, we obtain the following corollary, a generalization of \cite[Theorem 6]{AC14}.
\begin{corollary}\label{dilationN}
	Let $(\sigma, V^{(1)}, V^{(2)},\dots, V^{(k)})$  be a pure, completely contractive covariant representation of $\mathbb{E}$ on the Hilbert space $\mathcal{H}_{V},$ satisfying Brehmer-Solel condition. Then it is unitarily equivalent to the restriction of the induced representation $(\rho, S^{(1)}, \dots, S^{(k)})$ to QS $ \Pi_V \mathcal{H}_{V}$  of $\mathcal{F}(\mathbb{E})\otimes\mathcal{D}_{ *, V}$ where $\Pi_V:\mathcal{H}_{V}\to \mathcal{F}(\mathbb{E})\otimes_{\pi}\mathcal{D}_{ *, V}$ is an isometry.

\end{corollary} 
Furthermore it dilates the induced representation $(\rho, S^{(1)}, S^{(2)}, \dots, S^{(k)}),$ therefore there exists an isometry $\Pi_V:\mathcal{H}_{V}\to \mathcal{F}(\mathbb{E})\otimes_{\pi}\mathcal{H}$ satisfying Equation (\ref{dilate}).
Then $\mathcal{ K}=\Pi_{V}\mathcal{ H}_{V}$ is a QS of $\mathcal{F}(\mathbb{E})\ot \mathcal{H},$ therefore  $(\sigma, V^{(1)}, V^{(2)},\dots, V^{(k)})$ is isomorphic to  $(P_{\mathcal{ K}}\rho, P_{\mathcal{ K}}S^{(1)}, P_{\mathcal{ K}}S^{(2)}, \dots, P_{\mathcal{ K}}S^{(k)})|_{\mathcal{K}}.$

Note that if $k=1,$ then $\mathcal{K}$ will be a BQS under $(\rho, S)$ by Theorem \ref{hs}, i.e. there exist a Hilbert space $\mathcal{H}_{T},$ a pure isometric representation $(\mu, T)$ of $E$ on $\mathcal{H}_{T}$ and an isometric multi-analytic operator $M_\Theta: { \mathcal{H}_{T}} \to \mathcal{F}(E)\ot \mathcal{H}$ such that $\mathcal{ K}_{\Theta}= (M_{\Theta}(\mathcal{H}_{T}))^{\perp}$
\begin{equation*}
	(\sigma, V)\cong (P_{(M_{\Theta}(\mathcal{H}_{T}))^{\perp}} \rho, P_{(M_{\Theta}(\mathcal{H}_{T}))^{\perp}}S)|_{(M_{\Theta}(\mathcal{H}_{T}))^{\perp}}.
\end{equation*}  The space $\mathcal{ K}_{\Theta}$ is said to be {\it model space} corresponding to $(\sigma,V).$ Therefore, each pure, completely contractive covariant representation is unitarily equivalent to compression of the induced representation to the model space.

Let $(\sigma, V^{(1)}, \dots, V^{(k)})$  be a  completely contractive covariant representation of  $\mathbb{E}$ on  $\mathcal{H}_{V}.$ 
Now in the $k >1$ case, consider $M_{\Theta}:  \mathcal{H}_{V} \to\mathcal{F}(\mathbb{E})\ot \mathcal{ H}   $ be  an isometric multi-analytic operator. Let  $\mathcal{ K}_{\Theta}= (M_{\Theta}(\mathcal{H}_{V}))^{\perp}$ and  $\mathcal{S}_{\Theta}= M_{\Theta}(\mathcal{H}_{V})$  which are BQS and BS of $\mathcal{F}(\mathbb{E})\ot \mathcal{ H}$, respectively. For each $i \in I_k$, define  $$\pi_{\Theta}(a)=P_{\mathcal{K}_{\Theta}}\rho(a)|_{{\mathcal{K}_{\Theta}}}  \:\mbox{and} \:\:W_{\Theta}^{(i)}(\xi_{i})=P_{\mathcal{K}_{\Theta}}S^{(i)}(\xi_{i})|_{{\mathcal{K}_{\Theta}}},$$
for $\xi_{i} \in E_{i}$ and $a \in \mathcal{B}.$ Therefore  $W_{\Theta }=(\pi_{\Theta}, W_{\Theta}^{(1)}, W_{\Theta}^{(2)}, \dots, W_{\Theta}^{(k)} )$ is a completely contractive covariant representation of $\mathbb{E}$ on $\mathcal{K}_{\Theta}.$ Indeed, \begin{align*}
	(\phi^{(i)}(a)\ot I_{\mathcal{F}(\mathbb{E})\otimes \mathcal{H}})\widetilde{W}_{\Theta}^{ (i)*}&=(\phi^{(i)}(a)\ot I_{\mathcal{F}(\mathbb{E})\otimes \mathcal{H}})(I_{E_{i}}\ot P_{\mathcal{K}_{\Theta}})\widetilde{S}^{ (i)*}P_{\mathcal{K}_{\Theta}}\\&=(I_{E_{i}}\ot P_{\mathcal{K}_{\Theta}})(\phi^{(i)}(a)\ot I_{\mathcal{F}(\mathbb{E})\otimes \mathcal{H}})\widetilde{S}^{ (i)*}P_{\mathcal{K}_{\Theta}}\\&=(I_{E_{i}}\ot P_{\mathcal{K}_{\Theta}})\widetilde{S}^{ (i)*}\rho(a)P_{\mathcal{K}_{\Theta}}\\&=(I_{E_{i}}\ot P_{\mathcal{K}_{\Theta}})\widetilde{S}^{ (i)*}P_{\mathcal{K}_{\Theta}}\rho(a)P_{\mathcal{K}_{\Theta}}=\widetilde{W}_{ \Theta}^{(i)*}\pi_{\Theta}(a),
\end{align*} where the last equality follows by the fact that $\mathcal{K}_{\Theta}$ is a QS of ${\mathcal{F}(\mathbb{E})\otimes \mathcal{H}}.$ Now, using the definition of QS, we have \begin{align*}
	\sum_{{u}\subseteq{I_{k}}}(-1)^{|u|}\widetilde{W}_{\mathbf{e}(u)} \widetilde{W}_{\mathbf{e}(u)}^{*}=\sum_{{u}\subseteq{I_{k}}}(-1)^{|u|}P_{\mathcal{K}_\Theta}\widetilde{S}_{\mathbf{e}(u)}\widetilde{S}_{\mathbf{e}(u) }^{*}P_{\mathcal{K}_\Theta},
\end{align*} Therefore $(\sigma, V^{(1)}, \dots, V^{(k)})$ satisfies Brehmer-Solel condition if and only if $W_{\Theta}$  satisfies the  Brehmer-Solel condition whenever they are equivalent. One can now ask which type of covariant representation are unitarily equivalent to $W_{\Theta }$ on BQS, the following theorem answers this question.
\begin{theorem}
	Let $(\sigma, V^{(1)}, \dots, V^{(k)})$  be a  completely contractive covariant representation of  $\mathbb{E}$ on a Hilbert space $\mathcal{H}_{V}.$ Then the following are equivalent:\begin{enumerate}
		\item $(\sigma, V^{(1)}, \dots, V^{(k)})\cong (\pi_{\Theta}, W_{\Theta}^{(1)}, W_{\Theta}^{(2)}, \dots, W_{\Theta}^{(k)} )$ for some BQS $\mathcal{K}_{\Theta} \subseteq \mathcal{F}(\mathbb{E})\ot \mathcal{ H}$.
		\item  $(\sigma, V^{(1)}, \dots, V^{(k)})$ is pure and satisfies the Brehmer-Solel condition and \begin{equation*}
			(I_{E_{j}}\ot ( I_{E_{i}\ot P_{\mathcal{K}_{\Theta}}}- \wV^{(i) *}\wV^{(i)}))(t_{i,j} \ot I_{\mathcal{K}_{\Theta}})(I_{E_{i}}\ot  (I_{E_{j}\ot P_{\mathcal{K}_{\Theta}}}- \wV^{(j) *}\wV^{(j)}))=0,
		\end{equation*} for distinct $i\neq j.$ 
	\end{enumerate}
\end{theorem}
\begin{proof} 
	$1 \implies 2 $: Since $(\sigma, V^{(1)}, \dots, V^{(k)})\cong (\pi_{\Theta}, W_{\Theta}^{(1)}, W_{\Theta}^{(2)}, \dots, W_{\Theta}^{(k)} ),$ 
	we have a unitary map $U:\mathcal {H}_{V} \to
	\mathcal {K}_{\Theta}$ which gives the unitary equivalence of representations $\sigma$ and $\pi_{\Theta},$ and also for each $1\leq i \leq k,$ we have  \begin{equation}\label{ra}
		P_{\mathcal{K}_{\Theta}}\widetilde{S}^{ (i)}(I_{E_{i}}\ot P_{\mathcal{K}_{\Theta}})=\widetilde{W}_{\Theta}^{ (i)} = U  \widetilde{V}^{(i)}  (I_{E_{i}} \ot U^*),
	\end{equation} which is equivalent to \begin{equation*}
		U^{*}P_{\mathcal{K}_{\Theta}}\widetilde{S}^{ (i)}(I_{E_{i}}\ot P_{\mathcal{K}_{\Theta}}U)=   \widetilde{V}^{(i)}.
	\end{equation*}
	
	As by hypothesis $\mathcal{K}_{\Theta}$ is BQS, Theorem \ref{Beurlin} and Equation ({\ref{ra}}) follows that  \begin{align*}&
		0=	(I_{E_{j}}\ot ( (I_{E_{i}}\ot P_{\mathcal{K}_{\Theta}})- (({P_{\mathcal{K}_{\Theta}}\widetilde{S}^{ (i)}(I_{E_{i}}\ot P_{\mathcal{K}_{\Theta}})})^{*}P_{\mathcal{K}_{\Theta}}\widetilde{S}^{ (i)}(I_{E_{i}}\ot P_{\mathcal{K}_{\Theta}})))\\&(t_{i,j} \ot I_{\mathcal{K}_{\Theta}})(I_{E_{i}}\ot ( (I_{E_{j}}\ot P_{\mathcal{K}_{\Theta}})- ({P_{\mathcal{K}_{\Theta}}\widetilde{S}^{ (j)}(I_{E_{j}}\ot P_{\mathcal{K}_{\Theta}})})^{*}P_{\mathcal{K}_{\Theta}}\widetilde{S}^{ (j)}(I_{E_{j}}\ot P_{\mathcal{K}_{\Theta}})))\\&=(I_{E_{j}}\ot ( (I_{E_{i}}\ot P_{\mathcal{K}_{\Theta}})- (( U  \widetilde{V}^{(i)}  (I_{E_{i}} \ot U^*))^{*}U  \widetilde{V}^{(i)}  (I_{E_{i}} \ot U^*)))\\&(t_{i,j} \ot I_{\mathcal{K}_{\Theta}})(I_{E_{i}}\ot ( (I_{E_{j}}\ot P_{\mathcal{K}_{\Theta}})- (U  \widetilde{V}^{(j)}  (I_{E_{j}} \ot U^*))^{*} U  \widetilde{V}^{(j)}  (I_{E_{j}} \ot U^*)))\\&=	(I_{E_{j}}\ot ( I_{E_{i}\ot P_{\mathcal{K}_{\Theta}}}- \wV^{(i) *}\wV^{(i)}))(t_{i,j} \ot I_{\mathcal{K}_{\Theta}})(I_{E_{i}}\ot  (I_{E_{j}\ot P_{\mathcal{K}_{\Theta}}}- \wV^{(j) *}\wV^{(j)})).
	\end{align*} Also one can easily see that  \begin{align*}
		\sum_{u \subseteq I_k}(-1)^{|u|}\widetilde{V}_{\mathbf{e}(u)}\widetilde{V}_{\mathbf{e}(u)}^{*}\geq 0,
	\end{align*}using Equation (\ref{ra}), which completes the required condition for $ 2.$
	
	$2 \implies 1 $: We have given $(\sigma, V^{(1)}, \dots, V^{(k)})$ is pure and satisfies the Brehmer-Solel condition, therefore by using  Corollary \ref{dilationN},  we obtain \begin{equation*}
		(\sigma, V^{(1)}, \dots, V^{(k)})\cong (\pi_{\Theta}, W_{\Theta}^{(1)}, W_{\Theta}^{(2)}, \dots, W_{\Theta}^{(k)} )
	\end{equation*} for some QS $\mathcal{K}_{\Theta}.$ But \begin{equation*}
		(I_{E_{j}}\ot ( I_{E_{i}\ot P_{\mathcal{K}_{\Theta}}}- \wV^{(i) *}\wV^{(i)}))(t_{i,j} \ot I_{\mathcal{K}_{\Theta}})(I_{E_{i}}\ot  (I_{E_{j}\ot P_{\mathcal{K}_{\Theta}}}- \wV^{(j) *}\wV^{(j)}))=0,
	\end{equation*} for distinct $i\neq j,$  follows that $\mathcal{K}_{\Theta}$ is BQS.
	
\end{proof}
\section{Factorizations and invariant subspaces}\label{4}

This section establishes a relationship between invariant subspaces of compression of a pure isometric covariant representation with factorization of the corresponding inner function.

Let $(\sigma, V)$ and $(\mu, T)$ be pure isometric  representations of $E$ on Hilbert spaces $\mathcal{H}_{V}$ and $\mathcal{H}_{T},$ respectively. Let ${M}_\Theta:\mathcal{H}_{V}\to \mathcal{H}_{T}$ be an isometric multi-analytic operator.
Let us denote  $\mathcal{K}_{\Theta}=\mathcal{H}_{T}\ominus M_{\Theta} \mathcal{H}_{V},$ and $ M_{\Theta} \mathcal{H}_{V}$  be BQS and Beurling subspace of $\mathcal{H}_{T}$ corresponding to $\Theta,$ respectively.
Define  $$\pi_{\Theta}'(a):=P_{\mathcal{K}_{\Theta }}\mu(a)|_{{\mathcal{K}_{\Theta}}}  \:\mbox{and} \:\:W'_{\Theta}(\xi)=P_{\mathcal{K}_{\Theta}}T(\xi)|_{{\mathcal{K}_{\Theta}}},$$
for $\xi \in E$ and $a \in \mathcal{B}.$ Therefore $(\pi'_{\Theta}, W^{\prime}_{\Theta})$ is a completely contractive covariant representation of $E$ on $ \mathcal{K}_{\Theta }\subseteq\mathcal{H}_{T}.$

\begin{theorem}\label{Ber}
	Let $(\sigma, V)$ and $(\mu, T)$ be pure isometric  representations of $E$ on Hilbert spaces $\mathcal{H}_{V}$ and $\mathcal{H}_{T},$ respectively.  Suppose ${\mathcal{W}_{\mathcal{H}_{V}}}$  and ${\mathcal{W}_{\mathcal{H}_{T}}}$ are the generating wandering subspace for $(\sigma, V)$ and $(\mu,T),$ respectively. Let ${M}_\Theta:\mathcal{H}_{V}\to \mathcal{H}_{T}$ be an isometric multi-analytic operator. Then  the covariant representation $(\pi'_{\Theta}, W'_{\Theta})$ has an IS   if and only if there exist a Hilbert space $\mathcal{H}_{R},$ a pure isometric  representation $(\nu, R)$ of $E$ on $\mathcal{H}_{R}$, isometric multi-analytic operators ${M}_\Phi: {\mathcal{H}_{R}} \to {\mathcal{H}_{T}}$, ${M}_\Psi: {\mathcal{H}_{V}} \to {\mathcal{H}_{R}}$ such that  \begin{equation*}
		{\Theta}={\Phi}{\Psi}.
	\end{equation*}
\end{theorem}

\begin{proof}

	(	$\implies$)
	Let $\mathcal{S}$ be a nontrivial $(\pi'_{\Theta}, W'_{\Theta})$-IS of $\mathcal{K}_{\Theta} \subseteq \mathcal{H}_{T},$ then $\mathcal{S}\oplus M_\Theta \mathcal{H}_{V}$ is a $(\mu, T)$-IS of $\mathcal{H}_{T}.$ Since $(\mu, T)$ is pure,  using Theorem \ref{hs}, there exist a Hilbert space $\mathcal{H}_{R}$, a pure isometric  representation $(\nu, R) $ of $E$ on $\mathcal{H}_{R}$ and an isometric multi-analytic operator $M_\Phi: {\mathcal{H}_{R}} \to {\mathcal{H}_{T}}$ such that \begin{equation*}\label{quotient}
		\mathcal{S}\oplus M_\Theta \mathcal{H}_{V}=M_{\Phi}{\mathcal{H}_{R}}.
	\end{equation*}  Note that $M_{\Theta}\mathcal{H}_{V}\subseteq M_{\Phi}\mathcal{H}_{R}$ and $\mathcal{S}=M_{\Phi}\mathcal{H}_{R}\ominus M_{\Theta}\mathcal{H}_{V}.$   Then by Douglas's range  inclusion theorem \cite{D96}, there exists a contraction $Z: \mathcal{H}_{V} \to \mathcal{H}_{R}$ such that \begin{equation*}
		M_{\Theta}=M_{\Phi}Z.
	\end{equation*}  Using multi-analytic property of $M_{\Phi}$ and $M_{\Theta}$, we obtain \begin{equation*}
		M_{\Phi}Z\widetilde{V}=M_{\Theta}\widetilde{V}=\wT(I_{E}\ot M_{\Theta})=\wT(I_{E}\ot M_{\Phi} Z)=M_{\Phi}\widetilde{R}(I_{E} \ot Z).
	\end{equation*} Thus $\widetilde{R}(I_{E} \ot Z)=Z\widetilde{V}$ and  similarly $\nu(a)Z=Z\sigma(a),$ for $ a \in \mathcal{B},$ i.e.,   $Z$ is multi-analytic, therefore $Z=M_{\Psi}$, for some inner operator $\Psi: \mathcal{W}_{\mathcal{H}_{V}} \to {\mathcal{H}_{R}}.$ Indeed, ${\Theta}={\Phi}{\Psi}$ and  $M_{\Phi}$ and $M_{\Theta}$ are isometries, we deduce
	\begin{equation*}
		\|h\|=\|M_{\Theta}h\|=\|M_{\Phi}M_{\Psi}h\|=\|M_{\Psi}h\|,\:\:\: h\in \mathcal{ H}_{V}.
	\end{equation*}

	($\impliedby$)	
	On the other hand, suppose there exist a Hilbert space $\mathcal{H}_{R},$ a pure isometric representation $(\nu, R)$ of $E$ on $\mathcal{H}_{R}$ and  isometric multi-analytic operators  $M_\Psi: {\mathcal{H}_{V}} \to {\mathcal{H}_{R}}$ and $M_\Phi: {\mathcal{H}_R} \rightarrow  {\mathcal{H}_T}$ such that   \begin{equation*}
		{\Theta}={\Phi}{\Psi}.
	\end{equation*}  Since  ${\Theta}={\Phi}{\Psi},$ we have $M_{\Theta}(\mathcal{H}_V) \subseteq M_\Phi(\mathcal{H}_R).$  Define $\mathcal{S}:=M_{\Phi}(\mathcal{H}_R)\ominus M_{\Theta}(\mathcal{H}_V).$ Clearly $\mathcal{S}$ is a closed subspace of $\mathcal{K}_{\Theta}.$ Also note that $\mathcal{K}_{\Theta}\ominus \mathcal{S}=\mathcal{K}_{\Phi}.$ 
	Since $\widetilde{W'^{*}_{\Theta}}=(I_{E}\ot P_{\mathcal{K}_{\Theta}})\widetilde{T}^{*}P_{\mathcal{K}_{\Theta}},$ and $\mathcal{K}_{\Phi}\subseteq\mathcal{K}_{\Theta}$, we have $\mathcal{K}_{\Phi}$ is invariant under $ \widetilde{W'^{*}_{\Theta}} $.. Indeed let $ h \in \mathcal{K}_{\Phi}$ \begin{align*}
		\widetilde{W'^{*}_{\Theta}}h=(I_{E}\ot P_{\mathcal{K}_{\Theta}})\widetilde{T}^{*}P_{\mathcal{K}_{\Theta}}h=\widetilde{T}^{*}P_{\mathcal{K}_{\Theta}}h=\widetilde{T}^{*}h=\widetilde{T}^{*}P_{\mathcal{K}_{\Phi}}h= \widetilde{W'^{*}_{\Phi}}h\in \mathcal{K}_{\Phi},
	\end{align*} where  $\widetilde{W'^{*}_{\Phi}}=(I_{E}\ot P_{\mathcal{K}_{\Phi}})\widetilde{T}^{*}P_{\mathcal{K}_{\Phi}},$ Also in similar way  $\mathcal{S}^{\perp}$ is is invariant under $ \widetilde{W'^{*}_{\Theta}} .$ That is  $\mathcal{S}
	$ is an $(\pi'_{\Theta},W'_{\Theta})$-IS.
\end{proof} But this result is not true in general for multi-variable case (see \cite[Example 4.2]{BDDS}). With the help  of \cite[Theorem 1.1]{BDDS} classify factorizations of inner functions in terms of invariant subspaces of tuples of module operators \cite[Theorem 4.4]{BDDS}. We generalize this result for a pure doubly commuting isometric representation of product system on the Hilbert space. Let us denote  $\mathcal{K}_{\Theta}=\mathcal{H}_{T}\ominus M_{\Theta} \mathcal{H}_{V},$ and $ M_{\Theta} \mathcal{H}_{V}$  be BQS and Beurling subspace of $\mathcal{H}_{T}$ corresponding to $\Theta,$ respectively.
For each $i \in I_{k}$, define  $$\pi'_{\Theta}(a):=P_{\mathcal{K}_{\Theta}}\mu(a)|_{{\mathcal{K}_{\Theta}}}  \:\mbox{and} \:\:W_{\Theta}^{'(i)}(\xi_{i}):=P_{\mathcal{K}_{\Theta}}T^{(i)}(\xi_{i})|_{{\mathcal{K}_{\Theta}}},$$
for $\xi_{i} \in E_{i},$ $a \in \mathcal{B}$ and hence $(\pi'_{\Theta}, W_{\Theta}^{'(1)}, W_{\Theta}^{'(2)}, \dots, W_{\Theta}^{'(k)} )$ is a completely contractive covariant representation of $\mathbb{E}$ on $\mathcal{K}_{\Theta} \subseteq \mathcal{H}_{T}.$
We are now ready for the multi-variable analog of Theorem \ref{Ber}. The following theorem is a generalization of \cite[Theorem 4.4]{BDDS}.	  
\begin{theorem}\label{44}
	Let $(\sigma, V^{(1)}, \dots, V^{(k)})$ and $(\mu, T^{(1)}, \dots, T^{(k)})$  be pure DCI-representations of  $\mathbb{E}$ on Hilbert spaces $\mathcal{H}_{V}$ and $\mathcal{H}_{T},$ respectively. Let ${M}_\Theta:\mathcal{H}_{V}\to \mathcal{H}_{T}$ be an isometric multi-analytic operator. Then the following are equivalent:
	\begin{enumerate}
		\item There exist a Hilbert space $\mathcal{H}_{R},$ a pure DCI-representation $(\nu, R^{(1)},\dots,R^{(k)}) $ of $\mathbb{E}$ on $\mathcal{H}_{R}$ and isometric multi-analytic operators ${M}_\Phi: {\mathcal{H}_{R}} \to {\mathcal{H}_{T}}$, ${M}_\Psi: {\mathcal{H}_{V}} \to {\mathcal{H}_{R}}$ such that \begin{equation*}
			{\Theta}={\Phi}{\Psi}.
		\end{equation*}
		\item  There exists  $(\pi'_{\Theta}, W^{'(1)}_{\Theta},\dots, W^{'(k)}_{\Theta})$-IS $\mathcal{S}\subseteq \mathcal{K}_{\Theta}$  such that $\mathcal{S}\oplus\mathcal{S}_{\Theta}$ is a Beurling subspace of $\mathcal{H}_{T}.$
		\item There exists  $(\pi'_{\Theta}, W^{'(1)}_{\Theta},\dots, W^{'(k)}_{\Theta})$-IS $\mathcal{S}\subseteq \mathcal{K}_{\Theta}$ of $\mathcal{H}_{T}$ such that \begin{align*}&
			(I_{E_{j}}\ot ( (I_{E_{i}}\ot P_{\mathcal{K}_{\Theta} \ominus \mathcal{S}})- \widetilde{U}^{(i) *}\widetilde{U}^{(i)}))(t_{i,j} \ot I_{\mathcal{H}_{T}})\\&\:\:\:\:\:(I_{E_{i}}\ot ( (I_{E_{j}}\ot P_{\mathcal{K}_{\Theta} \ominus \mathcal{S}})- \widetilde{U}^{(j) *}\widetilde{U}^{(j)}))=0,
		\end{align*} where  $\widetilde{U}^{(i)}=P_{\mathcal{K}_{\Theta} \ominus \mathcal{S}}\widetilde{W'_{\Theta}}^{(i)}{_{|_{{E_{i}}\ot {\mathcal{K}_{\Theta} \ominus \mathcal{S}}}}},$ for each $i\in I_{k}.$
	\end{enumerate}
\end{theorem}

\begin{proof}

	$(1) \implies (2)$ : Suppose there exist  a Hilbert space $\mathcal{H}_{R},$ a pure DCI-representation $(\nu, R^{(1)},\dots,R^{(k)}) $ of $\mathbb{E}$ on $\mathcal{H}_{R}$ and  isometric multi-analytic operators ${M}_\Phi: {\mathcal{H}_{R}} \to {\mathcal{H}_{T}}$, ${M}_\Psi: {\mathcal{H}_{V}} \to {\mathcal{H}_{R}}$ such that $	{\Theta}={\Phi}{\Psi}.$ Therefore $M_{\Theta}\mathcal{H}_V \subseteq M_\Phi\mathcal{H}_R,$ we can define $\mathcal{S}:=M_{\Phi}\mathcal{H}_R\ominus M_{\Theta}\mathcal{H}_V.$ Clearly $\mathcal{S}$ is a closed subspace of $\mathcal{K}_{\Theta}.$ Also note that $\mathcal{K}_{\Theta}\ominus \mathcal{S}=\mathcal{K}_{\Phi}.$	Since $\widetilde{W'_{\Theta}}^{(i)*}=(I_{E_{i}} \ot P_{\mathcal{ K}_{\Theta}})\widetilde{T}^{(i)*}P_{\mathcal{K}_{\Theta}},$ and $\mathcal{K}_{\Phi}\subseteq\mathcal{K}_{\Theta}$, we have $\mathcal{K}_{\Phi}$  is invariant under $\pi'_{\Theta}$ and $ \widetilde{W'_{\Theta}}^{(i)*} $ for each $i\in I_{k}.$ Indeed, let $ h \in \mathcal{K}_{\Phi}$ \begin{align*}
		\widetilde{W'_{\Theta}}^{(i)*}h&=(I_{E}\ot P_{\mathcal{K}_{\Theta}})\widetilde{T}^{(i)*}P_{\mathcal{K}_{\Theta}}h=\widetilde{T}^{(i)*}P_{\mathcal{K}_{\Theta}}h=\widetilde{T}^{(i)*}h=\widetilde{T}^{(i)*}P_{\mathcal{K}_{\Phi}}h\\&= \widetilde{W'_{\Phi}}^{(i)*}h\in \mathcal{K}_{\Phi}.
	\end{align*} Also in similar way  $\mathcal{S}^{\perp}$ is invariant under$ \widetilde{W'_{\Theta}}^{(i)} $. That is  $\mathcal{S}
	$ is an $(\pi_{\Theta}, W^{'(1)}_{\Theta},\dots, W^{'(k)}_{\Theta})$-IS.
	
	$(2) \implies (1)$: Let $\mathcal{S}$ be a $(\pi'_{\Theta}, W^{'(1)}_{\Theta},\dots, W^{'(k)}_{\Theta})$-IS of $\mathcal{H}_{T}$ such that $\mathcal{S}\oplus M_\Theta \mathcal{H}_{V}$  is a DCS of $\mathcal{H}_{T}$. Therefore by using Theorem \ref{MT5}, there exist a pure DCI-representation $(\nu, R^{(1)},\dots,R^{(k)}) $ of $\mathbb{E}$ on $\mathcal{H}_{R}$ and an isometric multi-analytic operator $M_\Phi: {\mathcal{H}_{R}} \to {\mathcal{H}_{T}}$ such that \begin{equation}\label{quotient 11}
		\mathcal{S}\oplus M_\Theta \mathcal{H}_{V}=M_{\Phi}{\mathcal{H}_{R}}.
	\end{equation} Thus $M_{\Theta}\mathcal{H}_{V}\subseteq M_{\Phi}\mathcal{H}_{R},$ and hence by Douglas's range  inclusion theorem, there exists a contraction $Z: \mathcal{H}_{V} \to \mathcal{H}_{R}$ such that \begin{equation*}
		M_{\Theta}=M_{\Phi}Z.
	\end{equation*} Using multi-analytic property of $M_{\Phi}$ and $M_{\Theta}$, we get \begin{equation*}
		M_{\Phi}\widetilde{R}^{(i)}(I_{E} \ot Z)=\wT^{(i)}(I_{E}\ot M_{\Phi} Z)=\wT^{(i)}(I_{E}\ot M_{\Theta})=M_{\Theta}\widetilde{V}^{(i)}=M_{\Phi}Z\widetilde{V}^{(i)},
	\end{equation*} for each $ i \in I_{k}.$ Thus $\widetilde{R}^{(i)}(I_{E} \ot Z)=Z\widetilde{V}^{(i)}$, similarly $\nu(a)Z=Z\sigma(a),$ for $ a \in \mathcal{B}.$ Then there exists an inner operator $\Psi: \mathcal{W}_{\mathcal{H}_{V}} \to {\mathcal{H}_{R}}$ such that $Z= {M}_{\Psi}.$ Indeed, 	${\Theta}={\Phi}{\Psi}$ and using  $M_{\Phi}$ and $M_{\Theta}$ are isometries, we deduce
	\begin{equation*}
		\|h\|=\|M_{\Theta}h\|=\|M_{\Phi}M_{\Psi}h\|=\|M_{\Psi}h\|,\:\:\: h\in \mathcal{ H}_{V}.
	\end{equation*}

	$(1) \implies (3)$ : Suppose there exist a pure DCI-representation\\ $(\nu, R^{(1)},\dots,R^{(k)}) $ of $\mathbb{E}$ on $\mathcal{H}_{R}$ and isometric multi-analytic operators $M_\Phi: {\mathcal{H}_{R}} \to {\mathcal{H}_{T}}$, $M_\Psi: {\mathcal{H}_{V}} \to {\mathcal{H}_{R}}$ such that $	{\Theta}={\Phi}{\Psi}.$ Define $\mathcal{S} =M_{\Phi}\mathcal{H}_R\ominus M_{\Theta}\mathcal{H}_V$ and then $\mathcal{K}_{\Theta}\ominus \mathcal{S}=\mathcal{K}_{\Phi}.$  Therefore  $\mathcal{K}_{\Phi}$ is a BQS of $\mathcal{H}_{T}$  using Theorem \ref{Beurlin}, we get  \begin{equation*}
		(I_{E_{j}}\ot ( (I_{E_{i}}\ot P_{\mathcal{K}_{\Phi}})- \widetilde{U}^{(i) *}\widetilde{U}^{(i)}))(t_{i,j} \ot I_{\mathcal{H}_{T}})(I_{E_{i}}\ot ( (I_{E_{j}}\ot P_{\mathcal{K}_{\Phi}})- \widetilde{U}^{(j) *}\widetilde{U}^{(j)}))=0,
	\end{equation*} where  $\widetilde{U}^{(i)}=P_{\mathcal{K}_{\Phi}}\widetilde{W'_{\Theta}}^{(i)}{_{|_{{E_{i}}\ot {\mathcal{K}_{\Phi}}}}},$ for each $i\in I_{k}.$

	$(3) \implies (2)$: Since $\widetilde{W'_{\Theta}}^{(i)}=(I_{E}\ot P_{\mathcal{K}_{\Theta}})\widetilde{T}^{(i)*}|_{{\mathcal{K}_{\Theta}}},$ it follows that $\mathcal{K}_{\Theta}\ominus \mathcal{S}$  is a QS of $\mathcal{H}_{T}$. Therefore by hypothesis and using Theorem \ref{Beurlin}, $\mathcal{K}_{\Theta}\ominus \mathcal{S}$  is  a BQS. Now observe that \begin{align*}
		\mathcal{K}_{\Theta}\ominus \mathcal{S}= \mathcal{H}_{T}\cap (M_{\Theta}(\mathcal{H}_{V})\cup \mathcal{S})^{\perp}= \mathcal{H}_{T} \ominus (M_{\Theta}(\mathcal{H}_{V}) \oplus\mathcal{S}),
	\end{align*} which implies that $\mathcal{S}\oplus\mathcal{S}_{\Theta}$ is a Beurling submodule of $\mathcal{H}_{T}.$
\end{proof}
In the above theorem we discussed only non-trivial $(\pi'_{\Theta}, W^{'(1)}_{\Theta},\dots, W^{'(k)}_{\Theta})$-IS $\mathcal{S}\subseteq \mathcal{K}_{\Theta}$ of $\mathcal{K}_{\Theta}$. Let us discuss a trivial case:\begin{enumerate}
	\item First take $\mathcal{S} =\{0\}$. Recall that $\mathcal{S}=M_{\Phi}\mathcal{H}_R\ominus M_{\Theta}\mathcal{H}_V$  therefore $M_{\Phi}\mathcal{H}_R\ominus M_{\Theta}\mathcal{H}_V=\{0\}$, it follows that $M_\Phi\mathcal{H}_{R}\subseteq M_\Theta\mathcal{H}_{V}.$ Since 	${\Theta}={\Phi}{\Psi}$, $M_\Theta\mathcal{H}_{V}\subseteq M_\Phi\mathcal{H}_{R}.$ Hence \begin{align*}
		M_\Phi\mathcal{H}_{R}=M_\Theta\mathcal{H}_{V}= M_{\Phi}M_{\Psi}\mathcal{H}_{V}, 
	\end{align*} we obtain $M_{\Psi}\mathcal{H}_{V}=\mathcal{H}_{R}.$ Thus ${\Psi}$ is a unitary.
	\item Now take $\mathcal{S} =\mathcal{K}_{\Theta}$. Therefore \begin{align*}
		M_{\Phi}\mathcal{H}_R\ominus M_{\Theta}\mathcal{H}_V=\mathcal{K}_{\Theta}=\mathcal{H}_{T} \ominus M_{\Theta}\mathcal{H}_{V},
	\end{align*} it implies that  we obtain $ M_{\Phi}\mathcal{H}_R=\mathcal{H}_{T}.$ Thus ${\Phi}$ is a unitary.
\end{enumerate} Therefore from the above discussion, if $\mathcal{S}$ is  a non-trivial  $(\pi'_{\Theta}, W^{'(1)}_{\Theta},\dots, W^{'(k)}_{\Theta})$-IS of $ \mathcal{K}_{\Theta}$ if and only if isometric multi-analytic operators ${M}_\Phi: {\mathcal{H}_{R}} \to {\mathcal{H}_{T}}$, ${M}_\Psi: {\mathcal{H}_{V}} \to {\mathcal{H}_{R}}$ are not unitary.
The following corollary is a generalization of \cite[Corollary 4.5]{BDDS}.
\begin{corollary}
	Let $(\sigma, V^{(1)}, \dots, V^{(k)})$ and $(\mu, T^{(1)}, \dots, T^{(k)})$  be  pure DCI-representations of  $\mathbb{E}$ on  Hilbert spaces $\mathcal{H}_{V}$ and $\mathcal{H}_{T},$ respectively.  Let ${M}_\Theta: {\mathcal{H}_{V}} \to {\mathcal{H}_{T}}$ be an isometric multi-analytic operator. Then  isometric multi-analytic operators ${M}_\Phi: {\mathcal{H}_{R}} \to {\mathcal{H}_{T}}$, ${M}_\Psi: {\mathcal{H}_{V}} \to {\mathcal{H}_{R}},$ coming from Theorem (\ref{44}), are non-unitary if and only if the following holds: 
	\begin{enumerate}
		\item $\mathcal{S}$ is a non-trivial $(\pi'_{\Theta}, W^{'(1)}_{\Theta},\dots, W^{'(k)}_{\Theta})$-IS of $\mathcal{K}_{\Theta}.$ 
		\item $\mathcal{S}$ is not a Beurling subspace of $\mathcal{H}_{T}.$
		\item  $\mathcal{K}_{\Theta}\ominus \mathcal{S}$ does not reduce $(\mu, T^{(1)}, \dots, T^{(k)})$.
	\end{enumerate}  
\end{corollary}

\begin{proof}
	First assume isometric multi-analytic operators ${M}_\Phi: {\mathcal{H}_{R}} \to {\mathcal{H}_{T}}$, ${M}_\Psi: {\mathcal{H}_{V}} \to {\mathcal{H}_{R}}$ are non-unitary.  As by above discussion, if $\mathcal{S}$ is  a non-trivial  $(\pi'_{\Theta}, W^{'(1)}_{\Theta},\dots, W^{'(k)}_{\Theta})$-IS of $ \mathcal{K}_{\Theta}$ if and only if  isometric multi-analytic operators ${M}_\Phi$, ${M}_\Psi $  are not unitary. It proves (1).

	Let, if possible, $\mathcal{S}$ be a Beurling subspace of $\mathcal{H}_{T}$. Therefore by using Theorem \ref{MT5}, there exist a pure DCI-representation $(\mu_1, R_1^{(1)},\dots,R_1^{(k)}) $ of $\mathbb{E}$ on $\mathcal{H}_{R_1}$ and an isometric multi-analytic operator $M_{\Phi_1}: {\mathcal{H}_{R_1}} \to {\mathcal{H}_{T}}$ such that $\mathcal{S}=M_{\Phi_1}{\mathcal{H}_{R_1}}.$ But \begin{align}\label{inv}
		\mathcal{S}=M_{\Phi}\mathcal{H}_R\ominus M_{\Theta}\mathcal{H}_V=M_{\Phi_1}{\mathcal{H}_{R_1}},
	\end{align} we have $M_{\Phi_1}{\mathcal{H}_{R_1}} \subseteq M_{\Phi}\mathcal{H}_R,$ and hence by Douglas's range  inclusion theorem, there exists a contraction $Z: \mathcal{H}_{R_1} \to \mathcal{H}_{R}$ such that \begin{equation*}
		M_{\Phi_1}=M_{\Phi}Z.
	\end{equation*}	Using multi-analytic property of $M_{\Phi}$ and $M_{\Phi_1}$, we get \begin{equation*}
		M_{\Phi}\widetilde{R}^{(i)}(I_{E} \ot Z)=\wT^{(i)}(I_{E}\ot M_{\Phi} Z)=\wT^{(i)}(I_{E}\ot M_{\Phi_1})=M_{\Phi_1}\widetilde{R}_{1}^{(i)}=M_{\Phi}Z\widetilde{R}_{1}^{(i)},
	\end{equation*} for each $i \in I_{k}. $ Thus $\widetilde{R}^{(i)}(I_{E} \ot Z)=Z\widetilde{R}_{1}^{(i)}$, similarly $\nu(a)Z=Z\mu_1(a),$ for $ a \in \mathcal{B}.$ Then there exists inner operator $\Phi_2: \mathcal{W}_{\mathcal{H}_{R_1}} \to \mathcal{W}_{\mathcal{H}_{R}}$ such that $Z= {M}_{\Phi_{2}}.$ Indeed, $\Phi_1=\Phi\Phi_2$ and   $M_{\Phi}$ and $M_{\Phi_1}$ are isometries, we deduce
	\begin{equation*}
		\|h\|=\|M_{\Phi_1}h\|=\|M_{\Phi}M_{\Phi_2}h\|=\|M_{\Phi_2}h\|,\:\:\: h\in \mathcal{ H}_{R_{1}}.
	\end{equation*} By using Equation (\ref{inv}) \begin{align*}  \mathcal{S}=M_{\Phi_1}{\mathcal{H}_{R_1}}
		M_{\Phi}(M_{\Phi_2}{\mathcal{H}_{R_1}})=M_{\Phi}\mathcal{H}_R\ominus M_{\Theta}\mathcal{H}_V&=M_{\Phi}\mathcal{H}_R\ominus M_{\Phi}M_{\Psi}\mathcal{H}_V\\&=M_{\Phi}(\mathcal{H}_R\ominus M_{\Psi}\mathcal{H}_V),
	\end{align*} we obtain $M_{\Phi_2}{\mathcal{H}_{R_1}}=\mathcal{H}_R\ominus M_{\Psi}\mathcal{H}_V=\mathcal{K}_{\Psi}.$ It implies that $\mathcal{K}_{\Psi}$ is  \\$(\nu, R^{(1)}, \dots, R^{(k)})$-reducing subspace of $\mathcal{H}_{R}$ It implies that $\mathcal{K}_{\Psi}$ is trivial subspace which means that $\Psi$ is an unitary. But $\Psi$ is not a unitary operator; thus, it is a contradiction. Hence  $\mathcal{S}$ is not a Beurling subspace of $\mathcal{H}_{T}$. It proves (2).

	Suppose  $\mathcal{K}_{\Theta}\ominus \mathcal{S}$  reduces $(\mu, T^{(1)}, \dots, T^{(k)}).$ Then $(\mathcal{K}_{\Theta}\ominus \mathcal{S})^{\perp}= \mathcal{S}\oplus \mathcal{S}_{\Theta}$ is also reduces $(\mu, T^{(1)}, \dots, T^{(k)}).$ On the other hand, By using Equation (\ref{quotient 11})  $\mathcal{S}\oplus \mathcal{S}_{\Theta}=M_{\Phi}\mathcal{H}_{R},$ it follows that $\Phi$ is an unitary, which is contradiction as $\Phi$ is not unitary inner operator.
	
	For the converse part, suppose $\mathcal{S}$ is a non-trivial $(\pi_{\Theta}, W^{(1)}_{\Theta},\dots, W^{(k)}_{\Theta})$-IS of $\mathcal{K}_{\Theta}.$ Since $\Theta=\Phi\Psi$ and $\Theta$ is non constant, both $\Phi$ and $\Psi$ can not be unitaries. It remains to show that $\Phi$ and $\Psi$ can not be isometry operators. Let, if possible, $\Phi\equiv X_1$ for some non-unitary isometry $X_1$ and $\Psi$ is not unitary. Then \begin{align*}
		\mathcal{S}\oplus \mathcal{S}_{\Theta}=M_{\Phi}\mathcal{H}_{R}=X_1\mathcal{H}_{R}=\mathcal{H}_{X_1R}
	\end{align*} and hence $\mathcal{S}\oplus \mathcal{S}_{\Theta}$ reduces $(\mu, T^{(1)}, \dots, T^{(k)}),$ which contradicts to (3). On the other hand, if $\Psi\equiv X_2$ for some non-unitary isometry $X_2$ and $\Phi$ is not unitary. Then\begin{align*}
		\mathcal{S}= M_{\Phi}(\mathcal{H}_R\ominus M_{\Psi}\mathcal{H}_V)=M_{\Phi}(\mathcal{H}_R\ominus(X_2\mathcal{H}_V))=M_{\Phi}\mathcal{H}_{R-X_{2}V}
	\end{align*} is a Beurling subspace of $\mathcal{H}_{T},$ which contradicts to (2). It completes the proof. 
\end{proof}
\subsection*{Acknowledgment}
We thank Jaydeb Sarkar for making us aware of reference \cite{AC14}. Azad Rohilla is supported by a UGC fellowship (File No: 16-6(DEC.2017)
/2018(NET/CSIR)). Shankar Veerabathiran thanks ISI Bangalore for the visiting scientist position. Harsh Trivedi is supported by MATRICS-SERB  
Research Grant, File No: MTR/2021/000286, by the Science and Engineering Research Board
(SERB), Department of Science \& Technology (DST), Government of India.  We acknowledge the Centre for Mathematical \& Financial Computing and the DST-FIST grant for the financial support for the computing lab facility under the scheme FIST ( File No: SR/FST/MS-I/2018/24) at the LNMIIT, Jaipur.

\end{document}